%% file: paper.tex
\documentclass[12pt]{amsart}
\usepackage[utf8]{inputenc}
\usepackage{paralist}
\usepackage[left=3cm, right=3cm, top=3cm, bottom=3cm]{geometry}
\usepackage{hyperref}
\usepackage{lineno}

\usepackage{amsmath}
\usepackage{amssymb}

\def \figures{figures/}

\newtheorem{theorem}{Theorem}[section]
\newtheorem{lemma}[theorem]{Lemma}
\newtheorem{corollary}[theorem]{Corollary}
\newtheorem{conjecture}[theorem]{Conjecture}
\newtheorem{proposition}[theorem]{Proposition}
\theoremstyle{definition} %
\newtheorem{definition}[theorem]{Definition}
\newtheorem{remark}[theorem]{Remark}
\newtheorem{example}[theorem]{Example}
\newtheorem{examples}[theorem]{Examples}
\newtheorem{question}[theorem]{Question}
\newtheorem{fact}[theorem]{Fact}

\usepackage{mathrsfs}

\usepackage{ifthen}
\newboolean{draft}
\setboolean{draft}{false}
\newboolean{longversion}
\setboolean{longversion}{true}
\newboolean{inprogress}
\setboolean{inprogress}{false}

\usepackage{xspace}

\ifdraft
\usepackage{soul}
\usepackage[colorinlistoftodos]{todonotes}
\presetkeys{todonotes}{size=\small,color=green!30,inline,color=red!30}{}
\else
\newcommand{\todo}[2][]{}
\newcommand{\texthl}[1]{}
\fi
\newcommand{\nicolas}[1]{}
\newcommand{\justine}[1]{}

\ifinprogress
\newcommand{\tentative}[1]{}
\fi

\newcommand{\profile}[1]{\varphi_{#1}}
\newcommand{\profileseries}[1]{\mathcal{H}_{#1}}
\newcommand{\hilbert}[2][z]{\mathcal{H}(#2, #1)}
\newcommand{\QQ}{\mathbb{Q}}
\newcommand{\NN}{\mathbb{N}}
\newcommand{\ZZ}{\mathbb{Z}}
\newcommand{\KK}{\QQ}
\newcommand{\kerG}{T}
\newcommand{\kernel}[1]{\ker{#1}}
\newcommand{\restrict}[2]{{#1}_{|#2}}
\newcommand{\reynolds}[2][]{R_{#1}^{#2}}
\newcommand{\age}[1]{\mathcal{A}(#1)}
\newcommand{\orbitalgebra}[2][\KK]{#1\age{#2}}
\newcommand{\wreath}{\wr}
\newcommand{\sg}{\mathfrak{S}}
\newcommand{\sym}{\operatorname{Sym}}
\newcommand{\Even}{\operatorname{Even}}
\newcommand{\Gonblocks}{\overline G}
\newcommand{\Aut}{\operatorname{Aut}}
\newcommand{\Rev}{\ensuremath{\operatorname{Rev}}}
\newcommand{\suchthat}{\mid}
\newcommand{\hookdoubleheadrightarrow}{\hookrightarrow\mathrel{\mspace{-15mu}}\rightarrow}
\newcommand{\fix}{\operatorname{Fix}}
\newcommand{\stabilizer}{\operatorname{Stab}}
\newcommand{\fisubgroup}{K}
\newcommand{\Id}{\operatorname{Id}} %
\newcommand{\id}{\operatorname{id}} %
\newcommand{\canosyst}[1]{\mathcal{B}(#1)}
\newcommand{\finitesubsets}[1]{\mathscr{E}_{#1}}
\newcommand{\subsetalgebra}[1]{\QQ\finitesubsets{#1}}
\newcommand{\blocksystem}[1][B]{\mathcal{#1}}
\newcommand{\finiteblocksystem}[1][B]{\blocksystem[#1]^{<\infty}}
\newcommand{\blocklattice}[1]{\mathcal{L}(#1)}
\newcommand{\finiteblocklattice}[1]{\mathcal{L}^{<\infty}(#1)}
\newcommand{\infiniteblocklattice}[1]{\mathcal{L}^{\infty}(#1)}
\newcommand{\bound}{L}
\newcommand{\groupfiniteblocks}[3][]{[#2, #3^{\infty}\ifthenelse{\equal{#1}{}}{}{,#1}]}
\newcommand{\nested}[1]{\mathcal{B}_{\mathcal{B}}(#1)}
\newcommand{\nest}{\mathcal{B}_{\mathcal{B}}}
\newcommand{\blockofblocks}{superblock\xspace}
\newcommand{\blocksofblocks}{superblocks\xspace}

\newcommand{\stab}[2]{\text{Stab}_{#1}(#2)}

\newcommand{\BBindex}[1][j]{^{(#1)}}
\newcommand{\BB}[1][j]{BB\BBindex[#1]}
\renewcommand{\P}{\mathfrak{M}} %
\newcommand{\PM}{M} %
\newcommand{\Prim}[1][j]{\PM\BBindex[#1]}

\newcommand{\Data}{\operatorname{Data}}
\newcommand{\data}{\Delta}
\newcommand{\DataInv}{\operatorname{Group}}
\newcommand{\domain}{E}
\newcommand{\defi}{\textbf}
\newcommand{\finiteG}{G_{<\infty}}
\newcommand{\finiteK}{K_{<\infty}}

\usepackage{tkz-graph}
\GraphInit[vstyle = Shade]
\tikzset{
  LabelStyle/.style = { rectangle, rounded corners, draw,
                        minimum width = 2em, fill = yellow!50,
                        text = red, font = \bfseries },
  VertexStyle/.append style = { inner sep=5pt,
                                font = \Large\bfseries},
  EdgeStyle/.append style = {->, bend left} }

\usepackage{tikz,calc}
\usetikzlibrary{trees, patterns, calc, cd}

\newcommand{\trivialfull}[3][.2]{
\raisebox{0cm} {\begin{tikzpicture}
\draw [draw=blue!80, thick](3*4*#1+1.5*#1,0) ellipse (16*#1cm and 2*#2*#1cm);
  \foreach \x in {1,2,...,7}{
  	\foreach \y in {1,2,...,#2}{
      \draw[fill=black!50] (3*\x*#1 cm, 2*#1+\y*#1 cm) circle  (.4*#1cm);
    };
    \foreach \yy in {1,2,...,#3}{
      \draw[fill=black!50] (3*\x*#1 cm, -2*#1-\yy*#1 cm) circle  (.4*#1cm);
    };
  }
\draw[dotted, thick] (3*7.8*#1cm, .8*#2*#1 cm) -- (3*8.3*#1cm, .8*#2*#1 cm);
\draw[dotted, thick] (3*7.8*#1cm, -.8*#3*#1 cm) -- (3*8.3*#1cm, -.8*#3*#1 cm);
\end{tikzpicture}}}

\newcommand{\trivialfinest}[3][.2]{
\raisebox{0cm} {\begin{tikzpicture}
\draw [draw=white](3*4*#1+1.5*#1,0) ellipse (16*#1cm and 2*#2*#1cm);
  \foreach \x in {1,2,...,7}{
  	\foreach \y in {1,2,...,#2}{
      \draw[fill=black!50] (3*\x*#1 cm, 2*#1+\y*#1 cm) circle  (.4*#1cm);
    };
    \foreach \yy in {1,2,...,#3}{
      \draw[fill=black!50] (3*\x*#1 cm, -2*#1-\yy*#1 cm) circle  (.4*#1cm);
    };
  }
\draw[dotted, thick] (3*7.8*#1cm, .8*#2*#1 cm) -- (3*8.3*#1cm, .8*#2*#1 cm);
\draw[dotted, thick] (3*7.8*#1cm, -.8*#3*#1 cm) -- (3*8.3*#1cm, -.8*#3*#1 cm);
\end{tikzpicture}}}

\newcommand{\canonical}[3][.2]{
\raisebox{0cm} {\begin{tikzpicture}
\draw [draw=white](3*4*#1+1.5*#1,0) ellipse (16*#1cm and 2*#2*#1cm);
  \foreach \x in {1,2,...,7}{
    \draw [draw=blue!80, thick](3*\x*#1 cm,0cm) ellipse (#1 cm and 2*#2*#1 cm);
  	\foreach \y in {1,2,...,#2}{
      \draw[fill=black!50] (3*\x*#1 cm, 2*#1+\y*#1 cm) circle  (.4*#1cm);
    };
    \foreach \yy in {1,2,...,#3}{
      \draw[fill=black!50] (3*\x*#1 cm, -2*#1-\yy*#1 cm) circle  (.4*#1cm);
    };
  }
\draw[dotted, thick] (3*7.8*#1cm, .8*#2*#1 cm) -- (3*8.3*#1cm, .8*#2*#1 cm);
\draw[dotted, thick] (3*7.8*#1cm, -.8*#3*#1 cm) -- (3*8.3*#1cm, -.8*#3*#1 cm);
\end{tikzpicture}}}

\newcommand{\twoblocks}[3][.2]{
\raisebox{0cm} {\begin{tikzpicture}
\draw [draw=white](3*4*#1+1.5*#1,0) ellipse (16*#1cm and 2*#2*#1cm);
\draw [draw=blue!80, thick](2.8*4*#1+1.7*#1, .8*#2*#1 cm) ellipse (14*#1cm and .8*#2*#1cm);
\draw [draw=blue!80, thick](2.8*4*#1+1.7*#1,-.8*#3*#1 cm) ellipse (14*#1cm and .8*#3*#1cm);
  \foreach \x in {1,2,...,7}{
  	\foreach \y in {1,2,...,#2}{
      \draw[fill=black!50] (3*\x*#1 cm, 2*#1+\y*#1 cm) circle  (.4*#1cm);
    };
    \foreach \yy in {1,2,...,#3}{
      \draw[fill=black!50] (3*\x*#1 cm, -2*#1-\yy*#1 cm) circle  (.4*#1cm);
    };
  }
\draw[dotted, thick] (3*7.8*#1cm, .8*#2*#1 cm) -- (3*8.3*#1cm, .8*#2*#1 cm);
\draw[dotted, thick] (3*7.8*#1cm, -.8*#3*#1 cm) -- (3*8.3*#1cm, -.8*#3*#1 cm);
\end{tikzpicture}}}

\newcommand{\twoorbitsofblocks}[3][.2]{
\raisebox{0cm} {\begin{tikzpicture}
\draw [draw=white](3*4*#1+1.5*#1,0) ellipse (16*#1cm and 2*#2*#1cm);
  \foreach \x in {1,2,...,7}{
    \draw [draw=blue!80, thick](3*\x*#1 cm, .8*#2*#1 cm) ellipse (#1 cm and .7*#2*#1 cm);
    \draw [draw=blue!80, thick](3*\x*#1 cm, -.8*#3*#1 cm) ellipse (#1 cm and .7*#3*#1 cm);
  	\foreach \y in {1,2,...,#2}{
      \draw[fill=black!50] (3*\x*#1 cm, 2*#1+\y*#1 cm) circle  (.4*#1cm);
    };
    \foreach \yy in {1,2,...,#3}{
      \draw[fill=black!50] (3*\x*#1 cm, -2*#1-\yy*#1 cm) circle  (.4*#1cm);
    };
  }
\draw[dotted, thick] (3*7.8*#1cm, .8*#2*#1 cm) -- (3*8.3*#1cm, .8*#2*#1 cm);
\draw[dotted, thick] (3*7.8*#1cm, -.8*#3*#1 cm) -- (3*8.3*#1cm, -.8*#3*#1 cm);
\end{tikzpicture}}}

\newcommand{\superblock}[6]{ %
	\draw[draw=#6, thick] (#1-#5,#4) -- (#1-#5,#3);
	\draw[draw=#6, thick] (#1-#5,#3) arc (180:270:#5);
	\draw[draw=#6, thick] (#1,#3-#5) -- (#2,#3-#5);
	\draw[draw=#6, thick] (#2,#3-#5) arc (270:360:#5);
	\draw[draw=#6, thick] (#2+#5,#3) -- (#2+#5,#4);
	\draw[draw=#6, thick] (#2+#5,#4) arc (0:90:#5);
	\draw[draw=#6, thick] (#2,#4+#5) -- (#1,#4+#5);
	\draw[draw=#6, thick] (#1,#4+#5) arc (90:180:#5);
	}

\title[]{Classification of $P$-oligomorphic groups, conjectures of
  Cameron and Macpherson}

\author{Justine Falque \and Nicolas M. Thiéry}

\address{Univ Paris-Sud, Laboratoire de Recherche en Informatique,
Orsay; %
CNRS, Orsay, F-91405, France}


\begin{document}

\begin{abstract}
  Let $G$ be a group of permutations of a denumerable set
  $E$. The \textbf{profile} of $G$ is the function $\profile{G}$ which
  counts, for each $n$, the (possibly infinite) number $\profile{G}(n)$ of orbits of $G$
  acting on the $n$-subsets of $E$.
  Counting functions arising this way, and their associated generating
  series, form a rich yet apparently strongly constrained class. In
  particular, Cameron conjectured in the late seventies that, whenever
  the profile $\profile{G}(n)$ is bounded by a polynomial -- we say
  that $G$ is $P$-oligomorphic --, it is asymptotically equivalent to
  a polynomial. In 1985, Macpherson further asked whether the \textbf{orbit
    algebra} of $G$ -- a graded commutative algebra invented by
  Cameron and whose Hilbert function is $\profile{G}$ -- is finitely
  generated.

  In this paper we establish a classification of (closed)
  $P$-oligomorphic permutation groups in terms of \emph{finite
    permutation groups with decorated blocks}.

  It follows from the classification that the orbit algebra of any
  $P$-oligomorphic group is isomorphic to (a straightforward quotient
  of) the invariant ring of some finite permutation group. This
  answers positively both Cameron's conjecture and Macpherson's
  question. The orbit algebra is in fact Cohen-Macaulay; therefore the
  generating series of the profile is a rational fraction whose
  numerator has positive coefficients, while the denominator admits a
  combinatorial description.

  In addition, the classification provides a finite data structure for
  encoding closed $P$-oligomorphic groups. This paves the way for
  computing with them and enumerating them as well as for proofs by
  structural induction. Finally, the relative simplicity of the
  classification gives hopes to extend the study to, e.g., the class
  of (closed) permutations groups with sub-exponential profile.

  The proof exploits classical notions from group theory -- notably
  block systems of permutation groups and their lattice properties --,
  commutative algebra, and invariant theory.
\end{abstract}

\maketitle
\ifdraft
\clearpage
\tableofcontents
\clearpage
\fi

\section{Introduction}

\subsection{A conjecture of Cameron and a question of Macpherson}

Counting objects under a group action is a classical endeavor in
algebraic combinatorics. If $G$ is a permutation group acting on a
finite set $E$, Burnside's lemma provides a formula for the number of
orbits, while Pólya theory refines this formula to compute, for
example, the \textbf{profile} of $G$, that is, the function which
counts, for each $n\in\NN$, the number $\profile{G}(n)$ of orbits of
$G$ acting on subsets of size $n$ of $E$.

In the seventies, Cameron initiated the study of the profile when $G$
is instead a permutation group of an infinite set $E$. Of course the
question makes sense mostly if $\profile{G}(n)$ is finite for all $n$;
in that case, the group is called \textbf{oligomorphic}, and the
infinite sequence $\profile{G}=(\profile{G}(n))_n$ an \textbf{orbital
  profile}.

This setting includes, for example, counting integer
partitions (with optional length and height restrictions) or graphs up
to an isomorphism. It also relates naturally to logic: indeed
Ryll-Nardzewski's Theorem~\cite{Ryll-Nardzewski.1959} implies that
orbital profiles (of oligomorphic groups) are exactly the profiles of
$\aleph_0$-categorical theories, that is,
\emph{complete theories admitting a unique countable model}
(see \cite[Section 2.5]{Cameron.1990} for further details).

\iffalse
let indeed $T$ be a complete theory whose models are
countable, let $M$ be one such model (that is $M$ a relational
structure) and let $G$ be its automorphism group. Then, the age of $M$
(its finite restrictions) are given by the $G$-orbits of finite
subsets. By Ryll-Nardzewski's Theorem~\cite{Ryll-Nardzewski.1959}, $M$
is the unique model of $T$ (up to isomorphism) if and only if the
profile values are finite, that is, $G$ is oligomorphic.
\fi

The study of oligomorphic permutation groups has become a whole
research subject; see~\cite{Cameron.1990,Cameron.OPG.2009} for
surveys. One central topic is the description of general properties of
orbital profiles.
It was soon observed that the potential growth rates exhibited jumps.
For example, the profile either grows at least as fast as the
partition function, or is bounded by a
polynomial~\cite[Theorem~1.2]{Macpherson.1985}. In the latter case --
we say that $G$ is \defi{$P$-oligomorphic} -- it was conjectured to be
asymptotically polynomial:
\begin{conjecture}[Cameron~\cite{Cameron.1990}]
  \label{conjecture.cameron}
  Let $G$ be a $P$-oligomorphic permutation group. Then
  $\profile{G}(n)\sim a n^k$ for some $a>0$ and $k\in \NN$.
\end{conjecture}

As a tool in this study, Cameron introduced early on the \textbf{orbit
  algebra} $\orbitalgebra{G}$ of $G$, a graded connected commutative
algebra whose Hilbert function coincides with $\profile{G}$. As proved
in~\cite{Pouzet.2008.IntegralDomain}, it is always an integral domain
(assuming empty kernel), implying that the profile is non decreasing.

Macpherson asked the following:
\begin{question}[Macpherson~\cite{Macpherson.1985} p.~286]
  Let $G$ be a $P$-oligomorphic permutation group. Is
  $\orbitalgebra{G}$ finitely generated?
\end{question}
The point is that, by standard commutative algebra, whenever
$\orbitalgebra{G}$ is finitely generated, its Hilbert function is
asymptotically polynomial, as conjectured by Cameron. It is in fact \emph{eventually
  a quasi polynomial} (with a constant leading coefficient). 
  Equivalently, the
generating series of the profile
$\profileseries{G}=\sum_{n\in \NN}\profile{G}(n)z^n$ is a rational
fraction of the form
\begin{displaymath}
  \profileseries{G} = \frac {P(z)}{\prod_{i\in I}(1-z^{d_i})}\,,
\end{displaymath}
where $P(z)$ is a polynomial in $\ZZ[z]$ and the $d_i$'s are the degrees of the
generators.

\subsection{Main results}

In the extended abstract \cite{FalqueThiery.Macpherson.FPSAC}, we had
announced positive answers to both Cameron's conjecture and
Macpherson's question with sketches of proof.
In this paper we deliver a much stronger result with full proofs,
namely the extension of the classification of the five oligomorphic groups with
constant profile $1$~\cite{Cameron.1990} to a classification of all closed
$P$-oligomorphic groups (Theorem~\ref{theorem.classification}):
any such group $G$ is uniquely classified by a
finite piece of data $\Delta$, encoded as a finite permutation group
endowed with a block system whose blocks are suitably decorated. This
data can be extracted from $G$ by a detailed analysis of its lattice
of \emph{nested block systems} (block systems of block systems);
conversely, the group can be reconstructed from the data using wreath
products of finite permutation groups with one of the five (closed)
highly homogeneous groups, direct products, and diagonal actions of
finite permutation groups.

We then derive from the classification that the orbit algebra of any
$P$-oligomorphic group $G$ (closed or not\footnote{The profile and
  orbit algebra of an oligomorphic group and those of its closure
  coincide.}) is isomorphic to (a straightforward quotient of) the
invariant ring of some finite permutation group. By standard invariant
theory, the latter is finitely generated and even Cohen-Macaulay.

It follows that the generating series of the profile is of the form
\begin{equation}
\label{hilbert_series}
  \profileseries{G} = \frac {P(z)}{\prod_{i\in I}(1-z^{d_i})}\,,
\end{equation}
where the fraction is irreducible and $P$ is in $\NN[z]$; therefore
we have $\profile{G}(n)\sim an^{|I|-1}$ for some $a>0$.

\subsection{Connections with profiles and age algebras of relational structures}

This research is part of a larger program initiated in the seventies:
the study of the \emph{profile of relational
  structures}~\cite{Fraisse.TR.2000,Pouzet.2006.SurveyProfile} and in
general of the behavior of counting functions for hereditary classes of
finite structures, like undirected graphs, posets, tournaments,
ordered graphs, or permutations; see~\cite{Klazar.2008.Overview,Bollobas.1998.HereditaryPropertiesOfGraphs} for surveys.

In this program, \emph{jumps} in the set of potential growth rates of profiles are 
a ubiquitous phenomenon: not all growths are encountered.
For instance, undirected graphs~\cite{Balogh_Bollobas_Saks_Sos.2009} and
permutation classes~\cite{Kaiser_Klazar.2003} both exhibit
such jumps:
the counting functions, when bounded by a polynomial, are actually
asymptotically equivalent to a polynomial.
Besides,
by~\cite[Theorem~1.7]{Pouzet_Thiery.AgeAlgebra1}, the class of
relational structures with \emph{finite monomorphic dimension} --
rough analogues of transitive groups with a finite number of infinite
blocks -- also exhibits these jumps. Thanks to our proof of
Cameron's conjecture, the same is known to hold for homogeneous 
relational structures, which
correspond to oligomorphic permutation groups.
This, together with evidence from many examples,
suggests that a suitable generalization of the notion of block systems
may enable to prove that large classes of relational structures
exhibit such jumps in the possible profile growths.

In~\cite{cameron.1997}, Cameron extends the definition of orbit
algebra to the general context of relational structures. The
Cohen-Macaulay property holds when the profile is bounded
(see~\cite[Theorem~26]{Pouzet.2006.SurveyProfile} and
\cite[Theorem~1.5]{Pouzet_Thiery.AgeAlgebra1}); it can fail as soon as
the profile grows faster. Similarly, finite generation often fails;
however, when the monomorphic dimension is finite, there exists a
combinatorial characterization of when it
holds~\cite{Pouzet_Thiery.AgeAlgebra2}.

\subsection{Structure of the paper}

This paper is structured as follows. In
Section~\ref{section.preliminaries}, we review the basic definitions
of profiles and orbit algebras, and provide classical examples and operations.
We then recall some group theory tools that we will be using a lot, such
as the central notion of block systems and the subdirect product of 
groups, that will later help us handle interactions between the blocks
of a given system. %

In Section~\ref{section.canonical}, we see that block systems provide
lower bounds on the growth of the profile. This motivates the study of
block systems and their lattice properties, which leads to the
identification of a canonical system of blocks of blocks.

The classification of closed $P$-oligomorphic groups with a single
block of blocks is established in Section~\ref{section.finite_blocks}.
It was informed by an extensive computer exploration where
$P$-oligomorphic groups were approximated by finite groups. The proof
exploits towers of groups and subdirect products to suitably control
synchronizations, and encapsulates the most technical aspects of this
paper.

Section~\ref{section.classification_general_case} builds on the
previous one to extend the classification to all $P$-oligomorphic
groups, after establishing the existence of a minimal finite index
subgroup, which, together with a finite group acting diagonally,
generates $G$.

Finally, in Section~\ref{section.consequences}, we deduce from the
classification that the orbit algebra is (a trivial quotient of) an
invariant ring, providing positive answers to Cameron's conjecture and
Macpherson's question, and further deducing that it is Cohen-Macaulay.

\subsection{Perspectives}

A notable outcome of the classification is that it provides a uniform
finite description of any $P$-oligomorphic group; otherwise said a
data structure. This prepares the ground for algorithms and computations
with $P$-oligomorphic groups: constructing $P$-oligomorphic groups
from their classification, combining them together (cartesian or
wreath products), computing their properties such as the generating
series of the profile, constructing the orbit algebras and doing
arithmetic with their elements, etc. The classification also paves the
way for computations about the class of $P$-oligomorphic groups: there are
indeed countably many of them and we may iterate through them, count
them by kernel size and profile growth, etc. These algorithms are
being implemented in a library on top of SageMath~\cite{SageMath} and are the topic of
a follow-up publication; see the first author's PhD dissertation~\cite{Falque.thesis} in the meantime.

The classification also enables proofs by structural induction: if a
property holds for the five closed permutation groups with profile $1$
and is stable under wreath products, cartesian products and taking
finite index group extensions, then it holds for any closed
$P$-oligomorphic group. It holds more generally for any permutation
group obtained by such an inductive construction. This class of
oligomorphic permutation groups would be worth studying for its own
sake. At this stage, it seems plausible that it contains all
oligomorphic permutation groups with subexponential profile growth,
and that classifying those is within reach. On the other hand there
is little hope beyond exponential growth, due to the appearance of new,
more complicated primitive groups.

\subsection{Acknowledgments}

We would like to heartily thank Maurice Pouzet for suggesting to work
on this conjecture, and Peter Cameron and Maurice Pouzet for
enlightening discussions and benevolent encouragements throughout the
project.

This research was supported by extensive computer exploration, using
the open source mathematical systems GAP~\cite{GAP4} and
SageMath~\cite{SageMath}. We are grateful to their communities for
their continuous support, notably at the occasion of dissemination
events of the OpenDreamKit Horizon 2020 European Research
Infrastructures project (\#676541).

\section{Preliminaries}
\label{section.preliminaries}

\subsection{The age, profile and orbit algebra of a permutation group}
\label{subsection.profile.age}

Let $G$ be a permutation group, that is, a group of permutations of some
set $E$. Unless stated otherwise, $E$ is denumerable and $G$ is infinite.
The action of $G$ on the elements of $E$ induces an action on the set
$\finitesubsets{E}$ of finite subsets of $E$.

The \textbf{age} of $G$ is the set $\age{G}$ of the orbits
of finite subsets in $\finitesubsets{E}$ under this action. 
Within an orbit, all subsets share the
same cardinality, which is called the \textbf{degree} of the orbit. This
gives a grading of the age according to the degree of the orbits:
$\age{G} = \sqcup_{n\in \NN} \age{G}_n$; we will also use the notation 
$\age{G}^+ = \sqcup_{n\in \NN \backslash \{0\}} \age{G}_n$.
The \textbf{profile} of $G$ is
the function $\profile{G}: n\mapsto |\age{G}_n|$. In general, the
profile may take infinite values; the group is called
\textbf{oligomorphic} if it does not. 

We call \textbf{growth rate} of a profile bounded by a polynomial the smallest number
$r$ satisfying $\profile{G}(n) = O(n^r)$;
for instance, the growth rate of $n^2 + n$ is $2$. %
By extension, the \textbf{growth rate of a permutation group} $G$
is that of its profile $\profile{G}$.

\begin{definition}
  We say that a permutation group is \textbf{$P$-oligomorphic} if its
  profile is bounded by a polynomial.
\end{definition}

\begin{examples}
\label{example.basic_profiles}
  Let $G$ be the infinite symmetric group $\mathfrak{S}_{\infty}$. For
  each $n$ there is a single orbit containing all subsets of size $n$,
  hence $\profile{G}(n)=1$ for all $n$. We say that $G$ is \emph{highly
  homogeneous}, or, more informally, that it \emph{has profile $1$}.

  Now take $E=E_1 \sqcup E_2$, where $E_1$ and $E_2$ are two copies of
  $\NN$. Let $G$ be the group acting on $E$ by permuting the elements
  independently within $E_1$ and $E_2$ and by exchanging $E_1$ and
  $E_2$: $G$ is the \textbf{wreath product} $\sg_\infty \wreath \sg_2$. In that
  case, the orbits of subsets of cardinality $n$ are in bijection with
  the integer partitions of $n$ with at most $2$ parts.
  See Figure~\ref{figure.wreath_infinite_blocks} and
  Figure~\ref{figure.wreath_finite_blocks} for other examples of
  wreath products.
\end{examples}

\input{\figures figure_poly_sym}

Cameron's \textbf{orbit algebra} of $G$ is the graded connected vector
space $\orbitalgebra{G}$ of formal finite linear combinations of
elements of $\age{G}$; it is endowed with a commutative product as
follows: let $\subsetalgebra{E}$ be the \emph{set algebra} of $E$, namely
the vector space $\subsetalgebra{E}$ of (possibly infinite) formal
linear combinations of finite subsets of $E$, endowed with the
\emph{disjoint product} that maps two finite subsets to their union if
they are disjoint and to $0$ otherwise. Embed the orbit algebra
$\orbitalgebra{G}$ in the set algebra $\subsetalgebra{E}$ through the
linear morphism $i_G$ that maps each orbit to the sum of its elements.
The disjoint product stabilizes $i_G(\orbitalgebra{G})$ and can thus be
retracted to $\orbitalgebra{G}$. Some care needs to be taken at each
step to check that everything is well defined;
see~\cite{Cameron.1990,Cameron.OPG.2009} for details.

\iflongversion
\begin{example}
\label{example.basic_orbit_algebras}
\begin{itemize}
    \item Assume, as an exception, that $G$ is a finite permutation group. 
     Then its orbit algebra is finitely generated (likely with redundancy) by 
     its age, of which all elements are nilpotent of finite order (bounded by 
     the degree of the finite group $G$). It is thus of Krull dimension 0.
	\item Take $G = \sg_m$ for some $m \in \NN$. The profile is 1 until $n = m$
	and 0 beyond, and the orbit algebra is finitely generated as a vector
	space: it is isomorphic to $\QQ[y]$, where $y$ is nilpotent of order $m+1$.
	Take this to its infinite analogue, that is, $G = \sg_\infty$;
	then, $\orbitalgebra{G}$ is isomorphic to the algebra $\QQ[x]$ of
	univariate polynomials.
    Indeed, if $e_n$ denotes the unique orbit of all subsets of size $n$, then
    the morphism which maps $x^n$ onto $n!e_n$ is a morphism of algebras
    (use that, the way the orbital product goes, we have $e_k e_\ell = {{k+\ell}\choose{k}} e_{k+\ell}$). 
	\item Take on the other end the trivial permutation
      group on $E$. Each orbit consists of a single subset,
      and the profile counts for each $n$ the subsets of size $n$ of
      $E$. When $E$
      is infinite, the group is not at all $P$-oligomorphic nor even oligomorphic;
	  its orbit algebra is the free algebra generated by
      the infinitely many singletons in $E$.
\end{itemize}
\end{example}
\fi

\subsection{Basic properties and operations}

We recall here a few technical basics about orbit algebras and profiles,
in particular dealing with subgroups or restrictions, that one expects indeed
to be able to manipulate in a natural way.

\begin{lemma}[Relations between orbit algebras]\ 
  \label{lemma.subalgebras}
  \begin{enumerate}
  \item Let $G$ be a permutation group acting on $E$, and $F$ be a
    stable subset of $E$. Then, $\orbitalgebra{\restrict{G}{F}}$ is both a
    subalgebra and a quotient of $\orbitalgebra{G}$.
  \item Let $G$ be a permutation group acting on $E$, and $H$ be a
    subgroup, both of which being oligomorphic. Then, $\orbitalgebra{G}$ is a subalgebra 
    of $\orbitalgebra{H}$.
  \end{enumerate}
\end{lemma}

\begin{proof}
We exhibit the natural morphisms for each one of these cases.

\begin{enumerate}
	\item We have the following commutative diagram\\

\begin{center}
\begin{tikzcd}
\subsetalgebra{F} \ar[r,"\subseteq", hook] & \subsetalgebra{E}\\
\orbitalgebra{G_{|F}} \ar[u,"i_{G_{|F}}",hook] \ar[r,"\phi",hook]
\ar[d,"i_{G_{|F}}"',hook] &
\orbitalgebra{G} \ar[u,"i_G"',hook] 
\ar[l,"\psi",twoheadrightarrow,shift left=1ex] 
\ar[d,"i_G",hook] \\
\subsetalgebra{F} & \subsetalgebra{E} \ar[l,"\pi"',twoheadrightarrow]
\end{tikzcd}
\end{center}

where $\pi$ is the linear morphism mapping a subset of $E$ to itself if it is a subset 
of $F$ and to 0 otherwise. The injective morphisms $i_G$ and $i_{G_{|F}}$
are reversible where needed, which allows to define the respectively
injective and surjective morphisms $\phi$ and $\psi$ by composition.
	\item 
In the following diagram, $i_G$ and $i_H$ are the canonical embeddings of the orbit algebras into their set algebras.

\begin{center}
\begin{tikzcd}
 & \subsetalgebra{E} \\
 \orbitalgebra{G} \arrow[ur, hook, "i_G"] \arrow[r, "\phi"] & 
   \orbitalgebra{H} \arrow[u, hook, "i_H"']
\end{tikzcd}
\end{center}

The orbits of $G$ are unions of orbits (of same degree) of $H$, 
and since the groups are oligomorphic these unions are finite.
The image of $i_G$ is thus a subset of the image of $i_H$,
and therefore the diagram is commutative.\qedhere
\end{enumerate}
\end{proof}

\begin{lemma}[Direct product]\ 
  \label{lemma.operations.direct_product}
    Let $G$ and $H$ be permutation groups acting on $E$ and $F$
    respectively. Take $G\times H$ endowed with its natural action on 
    the disjoint union $E \sqcup F$.
    Then, $\age{G\times H}\simeq \age{G}\times\age{H}$, and
    $\orbitalgebra{G\times H} \simeq \orbitalgebra{G} \otimes
    \orbitalgebra{H}$; it follows that
    $\profileseries{G\times H}=\profileseries{G}\profileseries{H}$.
\end{lemma}

\begin{lemma}
  \label{lemma.subgroup.rate}
  Let $G$ be a permutation group and $\fisubgroup$ be a normal subgroup of
  finite index. Then,
  \begin{displaymath}
    \profile{G}(n) \ \leq\ \profile{\fisubgroup}(n) \ \leq\ |G:\fisubgroup| \profile{G}(n)\,.
  \end{displaymath}
  In particular, $\fisubgroup$ and $G$ share the same profile growth.
\end{lemma}

\begin{proof}
  Let $O$ be a $G$-orbit of elements.  Since $\fisubgroup$ is a normal 
  subgroup, $O$ splits into $\fisubgroup$-orbits on which $G$ -- 
  and actually $G/\fisubgroup$ -- acts transitively by
  permutation; there are thus finitely many such $\fisubgroup$-orbits, 
  all of the same size. In particular, infinite $G$-orbits split into infinite
  $\fisubgroup$-orbits, and similarly for finite ones.
\end{proof}

\subsection{Block systems and primitive groups}

A key notion when studying permutation groups is that of \textbf{block
  systems}; they are the discrete analogues of quotient modules in
representation theory. A \textbf{block system} is a partition of $E$
into parts, called \textbf{blocks}, such that each $g\in G$ maps blocks
onto blocks.

\begin{example}
Following is the list of all block systems of the cyclic permutation
group $\mathcal{C}_4$:
$\{\{1,2,3,4\}\}, ~\{\{1,3\}, \{2,4\}\}, ~\{\{1\}, \{2\}, \{3\}, \{4\}\}$.
\end{example}

  \newcommand{\diagonalblock}[1]{ \draw[draw=blue!100,
    thick](#1-2,0) ellipse (2cm and 5cm);
    \begin{scope}[xshift=#1cm]
      \draw[draw=blue, thick, rotate=45](0,0) ellipse (1.7cm and 5.6cm);
    \end{scope}
  }

\begin{center}
\begin{tikzpicture}

\node[circle, draw=black!0, fill=gray!40, scale=.6] (a1) at (0, 0){1};
\node[circle, draw=black!0, fill=gray!40, scale=.6] (a2) at (2, 0){2};
\node[circle, draw=black!0, fill=gray!40, scale=.6] (a3) at (2, 2){3};
\node[circle, draw=black!0, fill=gray!40, scale=.6] (a4) at (0, 2){4};
\draw[thick, draw=red!80](0, 0) circle (.35cm);
\draw[thick, draw=red!80](2, 0) circle (.35cm);
\draw[thick, draw=red!80](2, 2) circle (.35cm);
\draw[thick, draw=red!80](0, 2) circle (.35cm);
\draw[draw=blue, thick] (-.35,.35) arc (135:315:.5) ;
\draw[draw=blue, thick] (2.35,1.65) arc (-45:135:.5) ;
\draw[draw=blue, thick] (-.35,.35) -- (1.65,2.37);
\draw[draw=blue, thick] (.35,-.37) -- (2.35,1.65);
\begin{scope}[xshift=1cm, yshift=-1cm]
\begin{scope}[xshift=1cm, yshift=1cm, rotate=90]
\draw[draw=blue, thick] (-.35,.35) arc (135:315:.5) ;
\draw[draw=blue, thick] (2.35,1.65) arc (-45:135:.5) ;
\draw[draw=blue, thick] (-.35,.35) -- (1.65,2.37);
\draw[draw=blue, thick] (.35,-.37) -- (2.35,1.65);
\end{scope}
\end{scope}
\draw[draw=violet!70, thick] (0,-.68)--(2,-.68) arc (-90:0:.68) -- (2.68,2)  
arc (0:90:.68) -- (0,2.68) arc (90:180:.68) -- (-.68,0) arc (180:270:.68);
\end{tikzpicture}
\end{center}

The partitions $\{E\}$ and $\{\{e\} \suchthat e\in E\}$
are always block systems and are therefore called the trivial block systems.
A permutation group is \textbf{primitive} if it admits no non trivial
block system. By extension, an orbit of elements is \textbf{primitive} if
the restriction of the group to this orbit is primitive, and a block may
be called primitive if it cannot be refined into smaller non trivial blocks.

The following two theorems will be central in
our study.
\begin{theorem}[Macpherson~\cite{Macpherson.OIPG.1985} Theorem 1.1; see also~\cite{Cameron.1990} (3.21)]
  \label{primitive.hightly.homogeneous}
  The profile of an oligomorphic primitive permutation group is either
  the constant function $1$ or bounded below by an exponential.
\end{theorem}

In our study of Macpherson's question, all profiles are assumed to be
bounded by a polynomial, and therefore primitive groups always have profile
$1$: in other words, they are highly homogeneous. 
These groups are classified (up to \emph{closure}; see below).

\begin{theorem}[Cameron~\cite{Cameron.1990} (Section 3.4)]
  \label{profile.one.classification}
  There are only five closed highly homogeneous groups:
  \begin{compactenum}
  \item The automorphism group $\Aut(\QQ)$ of the rational chain
    (order-preserving bijections on $\QQ$);
  \item $\mathrm{Rev}(\QQ)$ (generated by $\Aut(\QQ)$ and a reflection);
  \item $\Aut(\QQ / \mathbb{Z})$, preserving the cyclic order (see $\QQ / \mathbb{Z}$ as a circle);
  \item $\mathrm{Rev}(\QQ / \mathbb{Z})$, generated by $\mathrm{Cyc}(\QQ / \mathbb{Z})$ and a reflection;
  \item $\sg_{\infty}$.
  \end{compactenum}
\end{theorem}

In the vocabulary of model theory, these are the groups preserving, respectively,
the dense linear order, the betweenness order, the circular order,
the separation relation, and a pure set.

The notion of closure refers here to the topology of simple
convergence, described in Section~2.4 of \cite{Cameron.1990}. Thanks
to the following classical lemma, it plays only a minor role for our purposes.

\begin{lemma} \label{lemma.completion_same_profile_and_orbits}
A permutation group and its closure share the same profile and orbit algebra.
\end{lemma}

\begin{proof}
  Let $G$ be a permutation group acting on a set $E$, and
take two finite subsets $A$ and $A'$ that are in the same orbit for its closure:
there exists a permutation $\sigma$ of the closure that maps $A$ to
$A'$. By definition of the closure (for the simple convergence),
there exists a sequence $(\sigma_i)_i$ of permutations  in $G$ that
coincide with $\sigma$ on more and more elements of $E$. Eventually,
$\sigma_i$ will coincide with $\sigma$ on $A$, and therefore $A$ and
$A'=\sigma_i(A)$ are in the same $G$-orbit.
The age is thus the same for $G$ and its closure, and with it the profile
and the orbit algebra.
\end{proof}

We make the following remark, that will prove crucial later on.

\begin{remark}
\label{remark.profile_one_fi_subgroups}
The set of the (closed) highly homogeneous permutation groups is
stable under taking finite index normal subgroups. Among those, three
out of five, $\Aut(\QQ)$, $\Aut(\QQ/\ZZ)$ and $\sg_\infty$, have no
proper finite index normal subgroup. We will refer to them as
\emph{the three minimal highly homogeneous groups}.
\end{remark}

\begin{lemma}
\label{lemma.sym_on_finite_blocks}
Let $G$ be a (closed) $P$-oligomorphic permutation group, endowed with a
block system.
If an infinite $G$-orbit of blocks is primitive and the blocks 
of the orbit are not singletons, 
then the action on these blocks is isomorphic to $\sg_\infty$.
\end{lemma}

\begin{proof}
  Using Theorem~\ref{profile.one.classification}, the action on the
  set of blocks is given by one of the five closed highly homogeneous
  groups.

  Assume first that $G$ acts as $\Aut(\QQ)$ on the blocks. Take
  $m\in \NN$. Choose an $m$-tuple of distinct blocks and, for each
  word $u=u_1\cdots u_m$ in the letters $1$ and $2$, choose a set
  $A_u$ of size $u_1+\cdots+u_m$ by picking, for each $i$, $u_i$
  elements from the $i$-th block. Since $G$ preserves the blocks and
  cannot swap their order, $(A_u)_u$ is a collection of $2^m$ non
  isomorphic sets of size between $m$ and $2m$. Hence
  $\sum_{n=m}^{2m} \profile{G}(n) \geq 2^m$. Taking $m$ large enough,
  we obtain a contradiction with $G$ being $P$-oligomorphic.

  The argument extends straightforwardly to the three other non
  symmetric highly homogeneous groups: the words $u$ just need to be
  considered up to a reflection, a cyclic permutation of the letters,
  or both; neither changes the exponential growth.
\end{proof}

\subsection{Wreath products}

Let $G$ and $H$ be permutation groups acting on $E$ and $F$
    respectively. Intuitively, the \textbf{wreath product} $G\wreath H$ acts on
    $|F|$ copies $(E_f)_{f\in F}$ of $E$, by permuting elements within each
    copy of $E$ independently according to $G$ and permuting the
    copies according to $H$. By construction, the partition
    $(E_f)_{f\in F}$ forms a block system, and $G\wreath H$ is not
    primitive (unless $G$ or $H$ is and $F$ or $E$, respectively, 
    is of size 1).

\begin{examples}[Algebras of wreath products]\ 
  \label{example.wreath_products}
  \begin{enumerate}
  \item Let $G$ be the wreath product $\sg_\infty \wreath \sg_k$. The
    profile counts integer partitions with at most $k$ parts.
    The orbit algebra is the algebra of symmetric polynomials over $k$
    variables, that is, the free commutative algebra with generators of
    degrees $1,\dots,k$. The generating series of the profile is given
    by $\profileseries{G} = \frac{1}{\prod_{d=1,\ldots,k} (1-z^d)}$.
    
    See also Figure~\ref{figure.wreath_infinite_blocks}, on which the
    red and blue subsets are in the same orbit. The associated
    integer partition $(3,1)$ can be read on the blue subset.
  \item Let $G'$ be a finite permutation group. Then, the orbit
    algebra of $G=\mathfrak{S}_{\infty} \wreath G'$ is isomorphic to the
    \textbf{invariant ring} $\KK [X]^{G'}$, which consists of the
    polynomials in $\KK[X]=\KK[X_1,\ldots,X_k]$ that are invariant
    under the action of $G'$.
  \item Let $G'$ be a finite permutation group. Then, the orbit
    algebra of $G= G'\wreath \mathfrak{S}_{\infty}$ is the free
    commutative algebra generated by the the set $\age{G'}^+$ of the 
    $G'$-orbits of non trivial subsets. The generating series of the profile is given by
    $\profileseries{G} = \frac{1}{\prod_{d} (1-z^d)}$, where $d$ runs
    through the degrees of $\age{G'}^+$, taken with multiplicity.
  \end{enumerate}
\end{examples}

\begin{proof}[Sketch of proof]
The first item is a special case of the second one, that we examine now.
Two subsets having the same number of elements in each infinite block
are in the same orbit, so if one canonically embeds the orbit algebra
of $G$ into the set algebra of $\domain$,
it is a subspace of that generated by the infinite sums 
\[
S_{\alpha} = \sum_{\text{card}(e \cap B_i) = \alpha_i} {\textit{\large e}} ~~,
\qquad \qquad \alpha = (\alpha_1, \ldots, \alpha_k) 
\]
for each multi-index $\alpha$ of length the number of infinite blocks
(that is the degree of $G'$), $B_i$ being the $i$-th 
block and $\textit{\large e}$ the subsets of $\domain$.
Now use the morphism $S_{\alpha} \mapsto X^{\alpha} = \prod_i X_i ^{\alpha_i}$
to embed $\orbitalgebra{G}$ into the algebra of polynomials.
The action of $G'$ on the blocks
acts the variables the same way, so the image of $\orbitalgebra{G}$
is the algebra of invariants of $G'$.\\

Set $G= G'\wreath \mathfrak{S}_{\infty}$ in order to prove the third item.
A canonical one-to-one correspondence can easily be established 
between $G$-orbits and multisets of $\age{G'}^+$: a finite subset of $\domain$ 
consists of a disjoint union of subsets that are included in the blocks, and thus 
each $G$-orbit is determined by the non trivial $G'$-orbits of these subsets, 
while the order does not matter. Since, the Hilbert series only depends on the
structure of graded vector space, it is then
$\frac{1}{\prod_d (1-z^{d})}$,
where the $d$'s are the orbital degrees of $\age{G'}^+$, the set of generators.

Define now an alternative notion of degree $\delta$,
the number of blocks involved in (the representatives of) an orbit.
In the orbit algebra $\orbitalgebra{G}$, the product of two orbits $O_1$ and
$O_2$ has one and only one dominant term for $\delta$, followed by lower 
degree terms (we say that $\delta$ is a \emph{filtration}, but we will
not go into the details about this notion). 
It is easy to see that the dominant term is the orbit that
corresponds to the multiset $\{O_1, O_2\}$.
Therefore, every $G$-orbit can be obtained as the dominant term of such a product
of $G'$-orbits in $\orbitalgebra{G}$, 
and since $\delta$ decreases on the other terms, it can actually 
be realized as a linear combination of products of $G'$-orbits (this can be
argued by induction on $\delta$).
Hence $\age{G'}^+$ generates all of $\orbitalgebra{G}$.

On the other hand, and by homogeneity, the shape of the Hilbert series 
imply that it is also a free family of elements: one can consider the
canonical morphism with the corresponding polynomial algebra (with
indeterminates of the degrees of $\age{G'}^+$), and 
deduce by dimension that it is an isomorphism.
\end{proof}

\input{\figures figure_partitions_blocks_finis}

We recall in the sequel some properties of wreath products that will
prove helpful later on.

\begin{proposition}
\label{proposition.wreath_quotient}
Let $F_1$, $F_2$, $P_1$ and $P_2$ be permutation groups such that $F_1$ (resp.
$P_1$) is a normal subgroup of $F_2$ (resp. $P_2$).
Then, $F_1 \wreath P_1$ is a normal subgroup of $F_2\wreath P_2$, 
and we have $(F_2 \wreath P_2) / (F_1 \wreath P_1) \simeq (F_2/F_1) \wreath (P_2/P_1)$.
\end{proposition}

\begin{proof}[Sketch of proof]
  Denote each element of $(F_2 \wreath P_2)$ as $g=((g_i)_i, p)$, with
  $i$ running through the domain of $P_2$, $g_i$ in $F_2$ for each
  $i$, and $p$ in $P_2$. Then check that the following defines a
  natural isomorphism:
  \begin{displaymath}
    \begin{cases}
      (F_2 \wreath P_2) / (F_1 \wreath P_1) &\longrightarrow (F_2/F_1) \wreath (P_2/P_1)\\
      g.(F_1\wreath P_1) &\longmapsto ((g_i.F_1)_i,\, p.P_1)
    \end{cases}\qedhere
  \end{displaymath}
\end{proof}

\begin{corollary}
\label{corollary.subwreath_index}
If $P_1$ and $P_2$ have finite degree $m$, we have
\[ [F_1 \wreath P_1 : F_2 \wreath P_2] = [F_1 : F_2]^m [P_1 : P_2]\,. \]
If the degree of $P_1$ and $P_2$ is infinite, $F_1 \wreath P_1$ is of infinite
index in $F_2\wreath P_2$ as soon as $F_1$ is a proper subgroup of $F_2$.
\end{corollary}

\begin{proof}
This is a direct consequence of Proposition~\ref{proposition.wreath_quotient}
and the fact that a wreath product $F\wreath P$ is of order 
$|F|^m |P|$ where $m$ is the degree of the permutation group $P$.
\end{proof}

\subsection{Subdirect products}
\label{subsection.subdirect_and_synchro}

The previous two subsections dwelled on particular, well described and
well understood kinds of $P$-oligomorphic groups. In
Section~\ref{section.canonical} we will see that the support $E$ of
any $P$-oligomorphic group $G$ splits up into pieces on each of which
$G$ essentially acts as one of those well understood groups. In the
later sections, we will recover the desired properties of $G$
(profile, orbit algebra) by controlling how the different pieces
interact together using the classical notion of \emph{subdirect
product}, which we recall here.

Let us start with the intuition. Assume that $E$ splits into
$G$-stable subsets $(E_i)_i$. The actions of $G$ on each $E_i$ are
not independent from each other in general: there may be partial or
full synchronization between the actions, which has consequences on
the profile and the orbit algebra. For two subsets this can be
formalized by describing $G$ as a subdirect product. The general case
can then be treated by induction.

\begin{definition}
 Let $G_1$ and $G_2$ be groups. A \textbf{subdirect product} of $G_1$ and 
 $G_2$ is a subgroup of $G_1 \times G_2$ which projects onto each 
factor under the canonical projections.
\end{definition}

For instance, assume $G$ is a permutation group that has two orbits
of elements $E_1$ and $E_2$. Then $G$ is a subdirect product of the
groups $G_1$ and $G_2$ induced respectively on $E_1$ and $E_2$.
Denote $N_1 = \fix_G (E_2)$ and $N_2 = \fix_G (E_1)$, the pointwise stabilizers of
$E_2$ and $E_1$, respectively. Then, $N_1$ and $N_2$ are normal
subgroups of $G$ and their intersection is trivial, so that we have 
$<N_1, N_2>  \simeq N_1 \times N_2$.

\begin{definition}
We call \textbf{synchronization} between $G_1$ and $G_2$ the following isomorphic
quotients (expressed after restriction of $N_i$ when needed):
\[\frac{G_1}{N_1} \simeq \frac{G}{N_1 \times N_2} \simeq \frac{G_2}{N_2} . \]
\end{definition}

Intuitively, these quotients describe the parts of each group that are
synchronized, whereas the $N_i$ are the independent parts.

\begin{proposition}
\label{prop.subdirect_characteristic}
Let $G$ be a subdirect product of $G_1$ and $G_2$. With the above notations,
we have \[G \simeq \{ (g_1, g_2) \in G_1\times G_2 \suchthat g_1 N_1 = g_2 N_2 \}. \]
\end{proposition}
This proposition implies that a permutation group arising as a
subdirect product is uniquely characterized by the associated groups
$G_1$, $G_2$, $N_1$ and $N_2$.

In addition, the possible synchronizations between two groups
are directly linked to their normal subgroups.
This observation combined with the classification of the groups of 
profile 1 (see Theorem~\ref{profile.one.classification}) and 
Remark~\ref{remark.profile_one_fi_subgroups},
lead to the following remark.
\begin{remark}
  \label{remark.profile_one_synchros}
  \label{remark.synchro_between_primitive_actions_on_blocks}
  Consider the action of a group $G$ on two infinite primitive orbits
  of points or of finite blocks. There is either:
  \begin{compactenum}
    \item no synchronization,
    \item total synchronization,
    \item a synchronized reflection in the cases of $\mathrm{Rev}(\mathbb{Q})$
    and $\mathrm{Rev}(\mathbb{Q}/\mathbb{Z})$ (synchronization \emph{of order 2}).
  \end{compactenum}
  On non trivial finite blocks, only the first two situations can
  occur by Lemma~\ref{lemma.sym_on_finite_blocks}.
\end{remark}

\iffalse
\begin{lemma}[Reduction 0]
  Synchronizations of order 2 between primitive orbits do not change
  the age of $G$ (this would be false for orbits of tuples).
\end{lemma}

This lemma implies that the synchronizations of order 2 can harmlessly be
ignored in the study of the orbit algebra of a group; so we will
assume from now on that only full synchronizations may exist.

This is also true regarding the infinite orbits of finite blocks 
(instead of just elements),
with [????] in mind:
taking two such orbits, either the permutations of their blocks 
fully synchronize (blockwise) or they
do not at all. We derive the following remark.

\begin{remark}
\label{remark.independance_in_canonical_block_system}
By construction, the actions of $G$ on the orbits of blocks of the canonical
 block system $\canosyst{G}$ are independent blockwise (potentially ignoring 
harmless synchronizations of order 2).
\end{remark}
\fi

\section{The nested block system}
\label{section.canonical}
Let $G$ be a $P$-oligomorphic permutation group. In this section, we
go back to the notion of block system and take a closer look at how
we can exploit it for our purposes. 
First, we show that each block system provides a lower bound on the growth of the
profile. Seeking to maximize this lower bound, we establish the nature of finite
lattices of the posets of block systems, and use it to derive a construction of a 
special ``block system'' (for an extended version of the notion) satisfying 
appropriate properties. The later sections will show that this so-called
nested block system minimizes synchronization and provides a tight lower bound.

\subsection{How block systems provide a lower bound on the profile}

We first consider the case where the block system is
\textbf{transitive}, that is, $G$ acts transitively on its
blocks. In this case, all the blocks are conjugated and thus share the
same cardinality.

\begin{lemma}
  \label{lemma.subalgebra_two_cases}
  Let $G$ be a $P$-oligomorphic permutation group,
  endowed with a transitive block system $\mathcal{B}$. Then,
  \begin{compactenum}
  \item Case 1: $\mathcal{B}$ has finitely many infinite blocks, as in
    Example~\ref{example.wreath_products} (1) and (2). Then $G$ is a
    subgroup of $\sg_{\infty} \wreath \sg_k$ (where $k$ is the number of
    blocks), and $\orbitalgebra{G}$ contains $\sym_k$ which is a free
    algebra with generators of degrees $(1,\dots,k)$.
  \item Case 2: $\mathcal{B}$ has infinitely many finite blocks, as in
    Example~\ref{example.wreath_products} (3). Then, $G$ is a subgroup
    of $G_{|B} \wreath \sg_{\infty}$, and $\orbitalgebra{G}$ contains
    the free algebra with generators of degrees given by that of the non trivial
    orbits of $G_{|B}$.
  \end{compactenum}
\end{lemma}

Note that the first case can be refined by stating that the orbit
algebra contains the algebra of invariants of the finite group $H$
acting on the blocks (which may be smaller than the full symmetric
group $\sg_k$); this algebra is typically not free.

\input{\figures figure_cas_essentiels}

\begin{proof}[Sketch of proof]
  The blocks share the same size by transitivity, and if their size (resp.
  number) is infinite, then their number (resp. size) has to be finite in order
  to keep the group $P$-oligomorphic (indeed, $G$ is otherwise a subgroup of
  $\sg_\infty\wreath \sg_{\infty}$ and its profile is bounded below by the 
  number of integer partitions).
  Use Lemma~\ref{lemma.subalgebras} and
  Examples~\ref{example.wreath_products} in each case.
\end{proof}

Assume that $G$ is endowed with a block system. Then, the proof of the
above lemma applies in the same fashion to the restrictions of $G$ to (the support of)
its orbits of blocks, leading the whole $\orbitalgebra{G}$ (with just one more use of 
Lemma~\ref{lemma.subalgebras}) to also contain the mentioned subalgebras (in the case
with finitely many finite blocks, refer instead to the first item of 
Example~\ref{example.basic_orbit_algebras}).
Recall also that in Case 2, we have the convenient property of 
Lemma~\ref{lemma.sym_on_finite_blocks}.

\begin{remark} \label{remark.independent_parts}
  Let $G$ be an oligomorphic permutation group, and $E_1,\ldots,E_k$ be
  a partition of $E$ such that each $E_i$ is stable under $G$.
  In our use case, we have a block system $\mathcal{B}$, and each $E_i$ is the
  support of one of the orbits of blocks in $\mathcal{B}$.

  Then, $G$ is a subgroup of
  $\restrict{G}{E_1}\times\cdots\times\restrict{G}{E_k}$ (precisely, it is a 
  \emph{subdirect product} of this direct product). Therefore,
  by Lemma~\ref{lemma.subalgebras}, $\orbitalgebra{G}$ contains
  $\orbitalgebra{\restrict{G}{E_1}}\otimes\cdots\otimes\orbitalgebra{\restrict{G}{E_k}}$
  as a subalgebra.
  In particular, the algebraic dimension of $\orbitalgebra{G}$ is
  bounded below by the sum of the algebraic dimensions of the
  $\orbitalgebra{\restrict{G}{E_i}}$.
  
  When, in addition, the actions of $G$ on each $E_i$ are completely
  independent, the containments above are equalities; then,
  $\orbitalgebra{G}$ is finitely generated if and only if each
  $\orbitalgebra{\restrict{G}{E_i}}$ is.
\end{remark}

\begin{remark}
\label{remark.lower_bound}
Combining Lemma~\ref{lemma.subalgebra_two_cases} and
Remark~\ref{remark.independent_parts}, each block system of $G$
provides a lower bound on the algebraic dimension of
$\orbitalgebra{G}$, and therefore on the growth rate of the profile.
\end{remark}

The following example illustrates that the lower bound on the profile
highly depends on the chosen block system.

\begin{example}
\label{example.treillis}
  Let $G=(\sg_2 \times \sg_2) \wreath \sg_\infty$. Following is the
  poset of all its block systems, ordered by refinement:

  \vspace{.5cm}
  \hspace*{-3cm}
  \begin{tikzcd}[column sep = tiny]
    & \trivialfull{2}{2} \ar[dl, dash] \ar[dr, dash] & \\
    \canonical{2}{2} & & \twoblocks{2}{2} & \\
    & \twoorbitsofblocks{2}{2} \ar[ul, dash] \ar[ur, dash] & \\
    & \trivialfinest{2}{2} \ar[u, dash] &
  \end{tikzcd}
  \vspace{.3cm}

  The picture below displays the lower bounds on the algebraic dimension
  that can be deduced respectively from each of these block systems,
  using Lemma~\ref{lemma.subalgebra_two_cases} and
  Remark~\ref{remark.independent_parts}:
  \begin{center}
    \begin{tikzcd}[column sep = tiny, row sep = tiny]
      & 1 \ar[dl, dash] \ar[dr, dash] & \\
      {\color{red}7} & & 2 & \\
      & 4 \ar[ul, dash] \ar[ur, dash] & \\
      & 1 \ar[u, dash] &
    \end{tikzcd}
  \end{center}
  For instance, for the block system with two orbits of blocks of size
  $2$, the lower bound on the algebraic dimension is $4=2+2$ since we have
  $G_{|B}=\sg_2$ in each orbit; for the block system with finite blocks
  of size $4$, the lower bound is $7$, for
  $G_{|B}=\sg_2 \times \sg_2$ has this many orbits of non empty
  subsets. This latter lower bound is obviously tight since the
  inclusion $G_{|B}\wreath \sg_\infty\subset G$ is an equality: the
  algebraic dimension of $\orbitalgebra G$ is $7$ and the growth rate
  of the profile of $G$ is $6$.
\end{example}

This example suggests that better lower bounds are obtained when
maximizing the size of the finite blocks (and then maximizing the
number of infinite blocks; consider also the example
$\sg_\infty \wreath \Id_n$ for that). 

Nevertheless, the bound provided by this heuristic alone can be improved 
at rather low cost, as advertised by the following example.

\begin{example}[Towards blocks of blocks]
\label{example.nested_system}
Consider the permutation group
$G=\mathcal{C}_4 \wreath (\sg_\infty \wreath \mathcal{C}_3)$; 
we use here the parentheses regardless of the associativity to emphasize 
the action of $G$ on its natural system of infinitely many 
(maximal) blocks of size $4$.

By Remark~\ref{remark.lower_bound}, this block system provides a lower 
bound of $4$ on the algebraic dimension.
As we will see, it is very crude; a lot of
information was lost when embedding the action on the blocks
$\sg_\infty \wreath \mathcal{C}_3$ into $\sg_\infty$. This action
was not even primitive to begin with: one can form $3$ infinite blocks
(of $4$-blocks). Let us exploit that information.
Consider the stabilizer $S$ of the three infinite blocks of finite blocks.
This is a normal subgroup of finite index of $G$, and therefore
it has the same algebraic dimension using Lemma~\ref{lemma.subgroup.rate}.
But now that these infinite blocks of blocks are stable parts of the domain, 
their contributions to the algebraic dimension can be treated separately;
which hands a bound of $3*4 = 12$ for $S$, and thus for $G$.

\iffalse
By associativity, we could alternatively write it as
$G=(\mathcal{C}_4 \wreath \sg_\infty) \wreath \mathcal{C}_3$; however
the former expression emphasizes the action of $G$ on its block system
$\blocksystem^{<\infty}$ of maximal finite blocks: namely $\blocksystem^{<\infty}$
consists of infinitely many maximal finite blocks, all of size $4$;
and it acts by $\sg_\infty \wreath \mathcal{C}_3$ on that block
system.

By Remark~\ref{3.3} this block system gives a lower bound of $4$ on
the algebraic dimension. As we will see, it is very crude: a lot of
information was lost when embedding $\sg_\infty \wreath \mathcal{C}_3$
into $\sg_\infty$: the action of $G$ on $\blocksystem^{<\infty}$ admits three
infinite blocks (of blocks). Let us instead exploit that information.
Consider the stabilizer of the three infinite blocks of finite blocks:
\begin{displaymath}
  \mathcal{C}_4 \wreath (\sg_\infty \wreath \Id_3) \approx (\sg_4 \wreath \sg_\infty)^3\,.
\end{displaymath}
This is a normal subgroup of finite index and, using
Lemma~\ref{lemma.subgroup.rate}, $G$ has the same algebraic dimension
of $12$. This provides a three times higher bound.
\fi
\end{example}

Let us step toward a generalization and a formalization of the 
phenomenon observed in the above example.

Let $G$ be a $P$-oligomorphic permutation group. Take a block system
$\mathcal{B}^{<\infty}$ of finite blocks only. Assume further that
these finite blocks are maximal: $G$ does not have any strictly
coarser system of finite blocks. This choice is motivated by the
earlier observations (see Example~\ref{example.treillis});
we will see in Subsection~\ref{subsection.lattices} that
$\mathcal{B}^{<\infty}$ always exists and is unique.

If $G$ has some finite orbits of elements, their union forms a stable
finite block which contains all other stable finite blocks (a stable
block is a union of orbits, that are obviously finite if the block
is).

By definition of a block system, $G$ acts on the set of blocks of 
$\mathcal{B}^{<\infty}$. Furthermore, this induced action
does not admit any non trivial finite block, for else 
the blocks of $\mathcal{B}^{<\infty}$ would not be maximal. It has no
special reason to be primitive though: its block systems will just
have infinite blocks only (plus possibly one singleton, if $\mathcal{B}^{<\infty}$
has a stable block) -- and finitely
many of them, for the same reasons as in Lemma~\ref{lemma.subalgebra_two_cases}.
By choosing one such system, we end up with two nested block systems: 
an inner one with finite blocks, and an outer one with (finitely many)
infinite blocks; 
in other words, \emph{a finite system of infinite blocks of finite blocks}.

\begin{remark}
\label{remark.lower_bound_from_nested}
A lower bound can be obtained from such a double, nested block system by first
stabilizing the infinite blocks of finite blocks (which does not change
the growth of the group, as stated by Lemma~\ref{lemma.subgroup.rate}),
and then applying the same method as in Remark~\ref{remark.lower_bound}.
For the same choice of (maximal) finite blocks, 
the lower bound $\bound$ provided by this method is better than the one 
deduced from the matching simple block system \emph{via} Remark~\ref{remark.lower_bound}.
\end{remark}

This is pretty much obvious if you consider the typical case highlighted
in Example~\ref{example.nested_system}.

As for the choice of the infinite blocks of blocks, the general intuition
remains the same as with classical, simple block systems: we feel that the
more the better, as far as the lower bound is concerned.

The next subsection formalizes these intuitions to construct a canonical
block system (in fact a system of blocks of blocks) that will hopefully
maximize the lower bound.

\subsection{Optimizing the lower bound through lattice structures}
\label{subsection.lattices}

As suggested by the previous subsection, maximizing the lower bound will
involve maximizing or minimizing block systems with certain
properties. To this end, we will exploit the lattice structure of the
poset of block systems, which we recall now.
\begin{proposition}
  \label{proposition.lattice_block_systems}
  Let $G$ be a permutation group $G$ acting on a set $E$, finite or
  infinite. The poset $\blocklattice{G}$ of all its block systems,
  endowed with the refinement order, is a sublattice of the lattice of
  set partitions of $E$. Its maximum and minimum are respectively the
  trivial block systems $\top=\{E\}$ and $\bot=\{\{e\}\suchthat e\in E\}$.
\end{proposition}
\begin{proof}
  Take two block systems $\blocksystem$ and $\blocksystem'$, and
  consider their meet in the lattice of set partitions, namely the set
  partition:
  \begin{displaymath}
    \blocksystem\wedge \blocksystem' = \{ B\cap B' \suchthat B\in \blocksystem \text{ and } B'\in \blocksystem' \}\,.
  \end{displaymath}
  It is straightforward to check that this is still a block system for
  the group. Hence this is the meet of $\blocksystem$ and
  $\blocksystem'$ in $\blocklattice{G}$.

  Similarly, consider the join $\blocksystem\vee \blocksystem'$ in the
  lattice of set partitions.
  It is obtained by taking the equivalence classes of the
  closure of the relation ``being in the
  same block in $\blocksystem$ or in $\blocksystem'$''. There remains
  to check that $\blocksystem\vee \blocksystem'$ is a block system: if
  $x$ and $y$ are in the same part and $\sigma$ is an element of $G$, then $\sigma(x)$
  and $\sigma(y)$ are in the same part as well. To this end, consider a
  sequence $x_0,\ldots,x_k$ such that we have $x_0=x$, $x_k=y$ and any two
  consecutive elements in the same block for either $\blocksystem$
  or $\blocksystem'$; then the same holds for the sequence
  $\sigma(x_0),\ldots,\sigma(x_k)$.

  In conclusion, $\blocklattice{G}$ is stable under both join and
  meet operations, and therefore a sublattice of the lattice of set 
  partitions of $E$.
\end{proof}

In the sequel, we will consider block systems with only finite blocks
(resp. only infinite blocks, up to kernel); the following propositions
state that those block systems form finite sublattices. This will
provide us with a canonical maximal (resp. minimal) block system from
which we will derive bounds.

\begin{proposition}
  \label{proposition.lattice_finiteblock_systems}
  Let $G$ be an oligomorphic permutation group, and
  $\finiteblocklattice{G}$ be the subposet of block systems consisting
  of finite blocks only. Then, $\finiteblocklattice{G}$ is a sublattice
  of $\blocklattice{G}$, with the trivial block system as minimum. If,
  in addition, $G$ is $P$-oligomorphic, then $\finiteblocklattice{G}$
  is finite, with a maximum $\finiteblocksystem$.
\end{proposition}

\begin{proposition}
  \label{proposition.lattice_infiniteblock_systems}
  Let $G$ be a $P$-oligomorphic permutation group, and
  $\infiniteblocklattice{G}$ be the subposet of block systems
  consisting of infinite blocks only; if the kernel of $G$ is non
  trivial, then finite blocks contained in the kernel are allowed as
  well. Then, $\infiniteblocklattice{G}$ is a finite sublattice of
  $\blocklattice{G}$, with a minimum and the trivial block system 
  as maximum.
\end{proposition}

Proving Propositions~\ref{proposition.lattice_finiteblock_systems}
and~\ref{proposition.lattice_infiniteblock_systems} will require a
couple of lemmas.

\begin{lemma}
  \label{lemma.increasing_on_lattices}
  Let $G$ be a $P$-oligomorphic group, and $\bound$ be
  the function that maps a block system onto the associated lower
  bound described in the previous subsection (Remark~\ref{remark.lower_bound}). 
  Let $\blocksystem < \blocksystem'$ be a cover (not involving the kernel)
  in the lattice
  $\finiteblocklattice{G}$, then we have
  $\bound(\blocksystem) < \bound(\blocksystem')$.

  If instead $\blocksystem < \blocksystem'$ is a cover (not involving the kernel)
  in the lattice $\infiniteblocklattice{G}$, we have
  $\bound(\blocksystem) > \bound(\blocksystem')$.
\end{lemma}

\begin{proof}
Assume that $\blocksystem < \blocksystem'$ is a cover in $\finiteblocklattice{G}$.
Pick one of the finite blocks $B$ in $\blocksystem'$ that splits into 
two new blocks $B_1$ and $B_2$ in $\blocksystem$; by conjugation, the 
same can be said about all the other blocks in the orbit of $B$.
There are two cases: either $B_1$ and $B_2$ may swap or
they may not. 

If not, then the support $O_B$ of the orbit of $B$ is the union of the supports 
$O_1$ and $O_2$ of the orbits of $B_1$ and $B_2$ (resp.), and the age of $G_{|O_B}$ 
contains the (disjoint) ages of the restrictions $G_1$ and $G_2$ to $O_1$ and $O_2$ 
(resp.). 
It also contains the additional orbits of subsets that have non empty 
intersections with both $B_1$ and $B_2$, so the inclusion is strict. 
Using Lemma~\ref{lemma.subalgebra_two_cases},
the provided bound on the profile is strictly better with the coarser
system (since we took a cover, the situation
in $\blocksystem$ and $\blocksystem'$ is the same everywhere else).

If $B_1$ and $B_2$ do swap, then we get a single orbit of (small) blocks 
in $\blocksystem$, just as in $\blocksystem'$; except that if one denotes by $H$
the restriction of $G$ to one of the small blocks in $\blocksystem$, the
restriction to one block of $\blocksystem'$ is $H \wreath \sg_2$, which
has a strictly larger age. Hence, $\blocksystem'$ provides a better bound.

As for $\infiniteblocklattice{G}$, the result is rather obvious from 
Lemma~\ref{lemma.subalgebra_two_cases}.
\end{proof}

\begin{lemma}
  \label{lemma.finiteblocklattice_stable_by_join}
  Let $G$ be an oligomorphic permutation group. The poset
  $\finiteblocklattice{G}$ is closed under taking joins (as defined in
  the lattice of set partitions).
\end{lemma}
The following simple example illustrates that this statement may fail
without the oligomorphic condition.
\begin{example}
  Recall that the (non oligomorphic) permutation group
  $\text{Aut}(\ZZ)$ is generated by the translation $x \mapsto x+1$,
  and take $G=\text{Aut}(\ZZ)\times\text{Aut}(\ZZ)$, acting on two
  copies of $\ZZ$: $E=\{1,2\} \times \ZZ$. This group admits an
  infinite family of block systems $(\blocksystem_j)_{j\in \ZZ}$ with
  non trivial finite blocks of size $2$:
  \begin{displaymath}
    \blocksystem_j := \{ \{(1,i), (2,i+j)\} \suchthat i\in\ZZ\}\,.
  \end{displaymath}
  The following picture illustrates the block systems $\blocksystem_0$
  and $\blocksystem_1$; their join is the trivial block system with a
  single infinite block.
  \newcommand{\intersectingblocks}[1]{ \draw[draw=red!50,
    very thick](#1-2,0) ellipse (2cm and 5cm);
    \begin{scope}[xshift=#1cm]
      \draw[draw=blue!100, thick, rotate=30](0,0) ellipse (1.7cm and 5.6cm);
    \end{scope}
  }
  \begin{center}
    \begin{tikzpicture}[scale=.2]
      \draw[draw=red!50, very thick](-2,0) ellipse (2cm and 5cm);
      \draw[draw=blue!100, thick, rotate=30](0,0) ellipse (1.7cm and 5.6cm);
      \node[circle, fill=gray!30, draw=black, scale=.7] (a) at (-2, 3.6){};
      \node[circle, fill=gray!30, draw=black, scale=.7] (a) at (-2, -3.6){};
      \draw[draw=red!50, very thick](2,0) ellipse (2cm and 5cm);
    \begin{scope}[xshift=4cm]
      \draw[draw=blue!100, thick, rotate=30](0,0) ellipse (1.7cm and 5.6cm);
      \node[circle, fill=gray!30, draw=black, scale=.7, opacity=.5] (a) at (-2, 3.6){};
      \node[circle, fill=gray!30, draw=black, scale=.7, opacity=.5] (a) at (-2, -3.6){};
      \end{scope}
      \draw[draw=red!50, very thick](6,0) ellipse (2cm and 5cm);
    \begin{scope}[xshift=8cm]
      \draw[draw=blue!100, thick, rotate=30](0,0) ellipse (1.7cm and 5.6cm);
      \node[circle, fill=gray!30, draw=black, scale=.7, opacity=.2] (a) at (-2, 3.6){};
      \node[circle, fill=gray!30, draw=black, scale=.7, opacity=.2] (a) at (-2, -3.6){};
      \end{scope}
      \intersectingblocks{12}
      \intersectingblocks{16}
      \intersectingblocks{20}
      \intersectingblocks{24}
      \intersectingblocks{28}
      \intersectingblocks{32}
      \intersectingblocks{36}
      \node (dots) at (46, 0){$\cdots$};
      \node (dots) at (58, 0){{\color{red!70}$\blocksystem_0$} and {\color{blue}$\blocksystem_1$}};
    \end{tikzpicture}
  \end{center}
  In general, the join of two of block systems $\blocksystem_i$ and 
  $\blocksystem_j$ with $i\ne j$ is composed of infinite blocks.
\end{example}
\begin{proof}[Proof of Lemma~\ref{lemma.finiteblocklattice_stable_by_join}]
Assume that the join of two systems of finite blocks $\blocksystem$
and $\blocksystem'$ from $\finiteblocklattice{G}$ contains at least one
infinite block. This block is thus a union of infinitely many blocks
from both $\blocksystem$ and $\blocksystem'$, in which every block from
one system intersects at least one block from the other one.
If all of the blocks of $\blocksystem$ involved were
singletons, each of them would be included in one block from $\blocksystem'$
and so the join would not have an infinite block; hence at least one
of them, call it $B_0$, is not.

Consider the stabilizer $S_0$ of this $B_0$ (in red in the center of 
Figure~\ref{figure.proof_stable_by_join}). 
In this subgroup, the union of the blocks from $\blocksystem'$ having a non
empty intersection with $B_0$ (in blue) is also stable, so as well is their
set difference with $B_0$. One can iterate the argument with the union
of blocks from $\blocksystem$ intersecting this stable domain (the outer crown
of red blocks), and so on. 

\begin{figure}[h]
\center
\includegraphics[scale=.3]{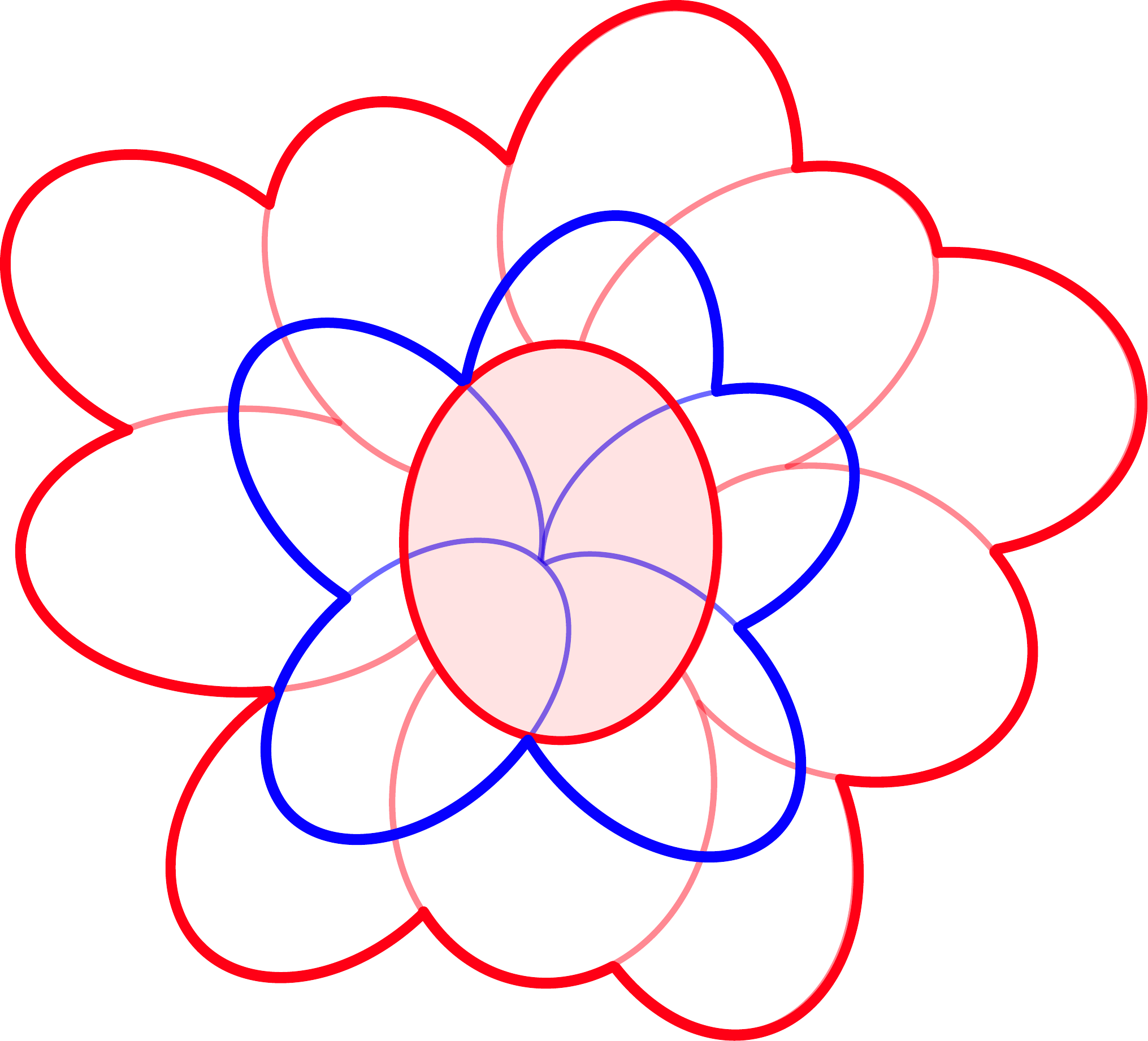}
\caption{Nested stable areas arising when stabilizing one block of $\blocksystem$}
\label{figure.proof_stable_by_join}
\end{figure}

This reveals an infinite sequence of finite disjoint (by taking 
the set difference every time) domains that are stable under the action
of $S_0$, and of which the first item is $B_0$. 
Take now two distinct subsets $A_1$ and $A_2$ of $E$, each of them 
consisting of two elements in 
$B_0$ and just one in any of the other $S_0$-stable domains.
An element of $G$ mapping $A_1$ to $A_2$, if there is any,
necessarily belongs to $S_0$, since the pair included in 
$B_0$ has no other choice but to be mapped onto the corresponding pair 
of $A_2$. Therefore, changing the $S_0$-stable domain in which we take
the singleton for $A_2$ (or $A_1$) exhibit infinitely many non isomorphic
subsets of size $3$ for $G$, which is to say infinitely many orbits of
degree $3$, and makes $G$ a non oligomorphic group.
\end{proof}

\begin{proof}[Proof of Proposition~\ref{proposition.lattice_finiteblock_systems}]
  Thanks to Lemma~\ref{lemma.finiteblocklattice_stable_by_join}, we already
  know that $\finiteblocklattice{G}$ is table under taking the join.
  We will successively prove that $\finiteblocklattice{G}$ is stable
  under meets, locally finite, and that it admits no infinitely
  increasing chain. We will then conclude that it is bounded and finite.

  Take two block systems $\blocksystem$ and $\blocksystem'$ in
  $\finiteblocklattice{G}$. Consider their meet in the lattice of block
  systems:
  \begin{displaymath}
    \blocksystem\wedge \blocksystem' = \{ B\cap B'\neq \varnothing \suchthat B\in \blocksystem \text{ and } B'\in \blocksystem' \}\,.
  \end{displaymath}
  By construction, it has again finite blocks, which proves that 
  $\finiteblocklattice{G}$ is stable under taking either
  joins or meets. In addition to this, the trivial block system
  $\bot=\{\{e\}\suchthat e\in E\}$ is obviously its minimal element.
  
  Let $\blocksystem$ be an element of $\finiteblocklattice{G}$.
  Consider the interval $[\bot, \blocksystem]$, and take a block
  system $\blocksystem'$ in that interval.
  The way a block $B$ in $\blocksystem$ splits into blocks
  in $\blocksystem'$ forces the way the blocks in the same orbit split
  in $\blocksystem'$ themselves. Since the blocks of $B$ are finite, and
  and there are finitely many orbits thereof for $G$ is oligomorphic,
  there are finitely many ways of
  splitting them. Therefore the interval $[\bot, \blocksystem]$ is
  finite, and the same holds for any interval:
  $\finiteblocklattice{G}$ is locally finite.

  Take a strict chain $C$ in $\finiteblocklattice{G}$. Using the local
  finiteness, embed this chain in a strict chain $C'$ where each step
  is a cover. Thanks to Lemma~\ref{lemma.increasing_on_lattices},
  $\bound$ is strictly increasing along that chain. Since $G$ is
  $P$-oligomorphic, $\bound$ is also bounded, and it follows
  successively that $C'$ and $C$ are finite.

  This ensures the existence of a maximum $\mathcal{B}^{<\infty}$, for else we could
  construct an infinite chain by starting with an element and then
  recursively take the join with an incomparable element. We
  conclude by remarking that $\finiteblocklattice{G}=[\bot, \mathcal{B}^{<\infty}]$
  is finite.
\end{proof}

\begin{proof}[Proof of Proposition~\ref{proposition.lattice_infiniteblock_systems}]

  The poset $\infiniteblocklattice{G}$ obviously has $\top=\{E\}$ as maximal
  element, and it is stable under joins: take indeed two block systems
  $\blocksystem$ and $\blocksystem'$ in $\finiteblocklattice{G}$:
  their blocks are infinite or included in the kernel. It is straightforward
  to check that the blocks of their join satisfy the same property.

  Let us prove that $\infiniteblocklattice{G}$ is stable under meet.
  Consider the meet of two block systems in $\infiniteblocklattice{G}$:
  \begin{displaymath}
    \blocksystem\wedge \blocksystem' = \{ B\cap B' \suchthat B\in \blocksystem 
    \text{ and } B'\in \blocksystem' \}\,.
  \end{displaymath}
  It has finitely many blocks. The union of all the finite ones is
  finite and stable under $G$; it is therefore included in the kernel of
  $G$. It follows that the blocks of
  $\blocksystem\wedge \blocksystem'$ are either infinite or included
  in the kernel of $G$, as desired.

  Consider an interval $[\blocksystem,\top]$ in
  $\infiniteblocklattice{G}$. Every block system $\blocksystem'$ from the interval
  is obtained by merging together some of the finitely many blocks of
  $\blocksystem$. Hence this interval is finite, and
  $\infiniteblocklattice{G}$ is locally finite.

  We conclude as in the proof of
  Proposition~\ref{proposition.lattice_finiteblock_systems}: there are
  no infinite chains in $\infiniteblocklattice{G}$ (a bit of care
  needs to be taken since $\bound(\blocksystem)$ may not be strictly
  increasing at steps where two finite blocks are merged; but there
  can be only finitely many such steps). This in turn ensures the
  existence of a minimal element $\mathcal{B}^\infty$ and the finiteness of
  $\infiniteblocklattice{G}$.

\end{proof}

\subsection{The nested block system}
\label{subsection.nested}

We may now use the structure of finite lattice on the block systems of a 
$P$-oligomorphic group to select a block system of a special kind 
(actually a system of blocks of blocks), which we expect to maximize the
associated lower bound, and thus to provide a best fitted set-up for the
study of the group.

\begin{definition}
  \label{theorem.canonical_block_system}
  Let $G$ be a $P$-oligomorphic permutation group. Take:
  \begin{compactenum}
  	\item the maximal (coarsest) element $\mathcal{B}^{<\infty}$ 
  of $\finiteblocklattice{G}$
	\item the minimal (finest) element $\mathcal{B}^\infty$
	of the lattice of block systems
	for the induced action of $G$ on $\mathcal{B}^{<\infty}$. 
  \end{compactenum}  
	We call the pair formed by the nested two partitions of $E$ defined 
	this way the \textbf{nested block system} 
	$\nested{G}$ of $G$.
  
\end{definition}

\begin{definition}
  We call an infinite primitive block of maximal finite blocks a \textbf{\blockofblocks}.
\end{definition}

Note that, under the preliminary assumption of maximality, 
the primitivity requirement is equivalent to asking that the infinite block
be minimal, so the above results on lattice structures imply that there is no
``choice'' for such \blocksofblocks -- as opposed to the many choices of blocks,
for the classical notion. From this point of view, the following (straightforward
from the construction process of the nested block system) proposition offers
an alternative definition of the \blocksofblocks, as the blocks of blocks
in $\nested{G}$.

\begin{proposition}[Structure of the nested system]
\label{proposition.structure_nested}
  The nested block system consists of finitely many \blocksofblocks, 
  and maybe one stable finite block. 
\end{proposition}

Besides providing a competitive lower bound on the profile growth, 
the nested block system can pride itself on some pleasant properties
of manageability.

\begin{lemma}
  \label{lemma.stable_superblocks_are_indep}
  Take two stable \blocksofblocks; the actions induced by $G$ on their sets of
  maximal finite blocks are independent (up to taking a normal subgroup 
  of finite index).
\end{lemma}

\begin{proof}
  By definition, the actions on the finite blocks of each \blockofblocks 
  are isomorphic to one of the five highly homogeneous groups. 
  Recall then Remark~\ref{remark.synchro_between_primitive_actions_on_blocks}, 
  and if need be take the finite index subgroup in which the actions of 
  type $\text{Rev}(\QQ)$ and $\text{Rev}(\QQ/\ZZ)$ are replaced by 
  $\text{Aut}(\QQ)$ and $\text{Aut}(\QQ/\ZZ)$ to avoid synchronizations of
  order $2$. Now the maximality of the finite blocks allows to eliminate the
  case of total synchronizations, which leaves none possible.
\end{proof}

Put otherwise, \blocksofblocks in $\nested{G}$ are not that far from 
independence, which would allow to use Remark~\ref{remark.independent_parts}.
This paper will eventually clarify what ``not that far'' actually means.

\begin{lemma}
  \label{lemma.finite_index_same_nested}
  Let $G$ be a $P$-oligomorphic permutation group, and $K$ be a finite
  index subgroup of $G$. Then we have
  $\nested{K}=\nested{G}$.
\end{lemma}

\begin{proof}
  We aim to prove that $K$ has the same superblocks as $G$.
  Observe first that blocks of imprimitivity of any permutation group
  are still blocks for any subgroup, as a direct consequence of the
  definition.
  Let $BB$ be a \blockofblocks, and $\P$ be the action of $G$ (implicitely
  after stabilization and restriction to the support of $BB$) on the set 
  of finite blocks of $BB$. Then, $\P$ is one of the five highly homogeneous
  groups; as the action of $K$ on the same set of finite blocks is necessarily
  a finite index subgroup of $\P$, it is highly homogeneous as well.
  We now just need to justify that the maximal finite blocks of $G$ are
  still maximal for $K$. Assume some of them are not, then
  there exists $m\geq 2$ \blocksofblocks $(BB\BBindex)_{1\leq j\leq m}$ 
  in $\nested{G}$ and an ordering of their respective finite blocks 
  $(B\BBindex_i)_i$ such that the unions $\cup_j B\BBindex_i$
  form new blocks for $K$ (up to taking the join in $\finiteblocklattice{K}$).
  This can only happen if some of the actions of $K$ on distinct $(B\BBindex_i)_i$
  fully synchronize for $1\leq j\leq m$. Since they are infinite 
  (highly homogeneous) actions, this is in contradiction with $K$'s 
  being of finite index. (Indeed, the action of $K$ on the blocks would be of
  infinite index in that of $G$, which is not possible.)
\end{proof}
  
The reader has probably already wondered at this point why to stop here.
We already have blocks of blocks, why not blocks of blocks of blocks, etc.?
The blocks of blocks of the nested block system allow a good description of
wreath products of type $F_1 \wreath P \wreath F_2$, where $F_1$
and $F_2$ are two finite groups that may be trivial and $P$ is an
infinite permutation group (recall Example~\ref{example.nested_system}). 
But what if we add
a layer of wreath product: $F_1\wreath P \wreath F_2 \wreath G$\,?
Well, it simply turns out that if $G$ is not finite then the group is not
$P$-oligomorphic anymore; and if it is, we are actually back to the same
configuration as earlier (by associativity).
Of course, if we had not made any hypothesis on the growth of the profile, it
would be relevant to consider any number of layers of blocks.

\section{Classification in the case of a single superblock}
\label{section.finite_blocks}
\label{hard_case}

In this section, we consider the class of closed $P$-oligomorphic
permutation groups $G$ with a single \blockofblocks, of which we
denote by $B_1,B_2,\ldots$ the maximal finite blocks. This class
includes wreath products $G=H \wreath \sg_\infty$ where $H$ is finite.
In Subsection~\ref{section.finite_blocks.family_examples} we construct
other examples by direct products; then, by combining wreath products
and direct products, we introduce a family of permutation groups that
subsumes all these examples. We show that their orbit algebras are
invariant rings of permutation groups, hence finitely generated and
Cohen-Macaulay.

In Subsection~\ref{section.finite_blocks.classification} we announce a
classification theorem: any instance of this class is isomorphic to
exactly one permutation group in the family. This answers positively
Macpherson's question for this class of permutation groups.

The next subsections undertake the proof
of the classification theorem: Subsection~\ref{section.finite_blocks.action} handles the
action on the set of blocks; Subsection~\ref{section.finite_blocks.tower} 
introduces the tower of $G$ in order to deal with the
action within the blocks, an object that will be the key tool in the rest 
of the proof, and that turns out to be classified; 
finally, Subsection~\ref{section.finite_blocks.lift} shows that this classification can be
lifted to the groups themselves.

\subsection{A family of examples beyond wreath products}
\label{section.finite_blocks.family_examples}

\begin{definition}
  We call \textbf{direct product on blocks} of two permutation groups
  $H$ and $S$ and denote by $H \mathbin{\square} S$ the permutation group defined by 
  the action of $H \times S$ on $\text{deg}(S)$ blocks of size $\text{deg}(H)$ by 
  \begin{displaymath}
    b_{r,i} . (\tau, \sigma) = b_{\tau(r),\sigma(i)}\,,
  \end{displaymath}
  where $b_{1,i},\ldots b_{m,i}$ is an arbitrary ordering of the elements
  of each block $B_i$. It is isomorphic to the natural action of
  $H \times S$ on the cartesian product of the supports.
\end{definition}

This can be pictured as $H$ and $S$ acting respectively by permutation
of the rows and of the columns of a (potentially infinite) matrix.

As opposed to the wreath product, where $H$ acts independently on each
block, here $H$ acts diagonally on all blocks at once. These two cases
are in this regard the two opposite ends of the spectrum of all
possible synchronizations between blocks.

It is then natural to think of a class of groups that would complete
the spectrum. We introduce such groups, as hybrids of wreath products
and direct products.

\begin{definition}
  Let $H \triangleleft H_0$ and $\P$ be three permutation groups, with $H$ 
  and $H_0$ finite.
  Denote by $\groupfiniteblocks[\P]{H_0}{H}$
  the permutation group generated by the elements of $H\wreath \P$ and
  $H_0 \mathbin{\square} \P$. For short, denote by
  $\groupfiniteblocks{H_0}{H}=\groupfiniteblocks[\sg_\infty]{H_0}{H}$.
\end{definition}

\begin{remark}
  \label{remark.classification_finite_blocks}
  The group $\groupfiniteblocks[\P]{H_0}{H}$ is $P$-oligomorphic if
  and only if we have $\P=\sg_\infty$ or $H_0=H=\Id_1$: indeed,
  $\groupfiniteblocks[\sg_\infty]{H_0}{H}$ contains $H\wreath \sg_\infty$ as a
  subgroup; it is therefore $P$-oligomorphic; the other implications
  are trivial using Lemma~\ref{lemma.sym_on_finite_blocks}.
\end{remark}

\begin{lemma}
  \label{lemma.subwreath}
  The permutation group $G=\groupfiniteblocks[\P]{H_0}{H}$ contains
  $H\wreath \P$ as a normal subgroup of finite index $[H:H_0]$.
  In addition, we have $G=(H\wreath \P)H_0$.
\end{lemma}
\begin{proof}
  First note that $G$ can be defined equivalently as the group
  generated by $H\wreath \Id_\infty=\,<H^\infty$, $\Id\wreath \P>$, and the
  finite group $H_0 \mathbin{\square} \Id_\infty$. For the sake of notations, and when
  there is no ambiguity, we identify an element $h_0$ of $H_0$ with
  the element $(h_0,h_0,\ldots)$ of $H_0 \mathbin{\square} \Id_\infty$, and identify
  $H_0$ with $H_0 \mathbin{\square} \Id_\infty$.

  Note that $h_0$ commutes with the elements of $\Id\wreath \P$ and,
  by normality of $H$ in $H_0$, skew-commutes with those of
  $H\wreath \Id_\infty$\,, meaning $H^\infty h_0=h_0H^\infty$. It follows that
  we have
  \begin{displaymath}
    G = \bigcup_{h_0\in H_0} (H\wreath\P).h_0\,.
  \end{displaymath}
  This union becomes a decomposition into cosets if the range is
  restricted to some collection of representatives of the cosets of
  $H$ in $H_0$. Therefore $H\wreath \P$ is normal and of finite
  index $[H:H_0]$ in $G$, as desired.
\end{proof}

We now describe the orbit algebra of $\groupfiniteblocks[\P]{H_0}{H}$
as an invariant ring of a finite permutation group. Recall that the orbit
algebra $\orbitalgebra{H\wreath\P}$ of $H\wreath\P$ is the free
commutative algebra $\KK[X]$, with
$X=(X_{\overline{A}})_{\overline{A}}$ where $\overline A$ ranges through
the non-trivial $H$-orbits, and $X_{\overline{A}}$ denotes the
$H\wreath \P$\,-orbit of $A$, seen as an element of the orbit algebra.
Finally, lift the action of $H_0$ on the $H$-orbits $\overline{A}$ to
an action on the variables $X_{\overline{A}}$.
\begin{proposition}
  \label{proposition.wreath_direct_orbit_algebra}
  With the above notations, the orbit algebra $\orbitalgebra{G}$ of
  $G=\groupfiniteblocks[\P]{H_0}{H}$ is isomorphic to the invariant
  ring $\KK[X]^{H_0}$.
\end{proposition}
\begin{proof}
  That an element $h_0\in H_0$ and $A$ a subset of the support of $H$.
  Check that the $H\wreath\P$-orbit of $A$ is mapped onto another such
  $H\wreath\P$-orbit, as prescribed by the announced action of $H_0$
  on the variables $X_{\overline{A}}$.
\end{proof}

\begin{remark}
  The variables of invariant rings are commonly taken of degree $1$;
  this is not the case here: the degree of the variable
  $X_{\overline{A}}$ is given by $|A|$. This must be taken into
  account when computing the Hilbert series using Molien's formula or
  Pólyà enumeration.
\end{remark}

\subsection{Classification and application to Macpherson's conjecture}
\label{section.finite_blocks.classification}

We may now state the main theorem of this section, which includes the
classification of
the trivial case of highly homogeneous groups $\P$ (case
$H_0 = H = \Id_1$ below); nevertheless, the core of this section is
about the case of non trivial finite blocks, in which $\P$ is
necessarily $\sg_\infty$.

\begin{theorem}[Classification on one \blockofblocks]
  \label{theorem.classification_finite_blocks}
  Let $G$ be a closed $P$-oligomorphic permutation group
  such that $\nested{G}$ consists of a single \blockofblocks.
  Then $G$ is isomorphic as a permutation
  group to $\groupfiniteblocks[\P]{H_0}{H}$, where
  $H \triangleleft H_0$ are two finite permutation groups and $\P$ is
  one of the five (closed) highly homogeneous groups. In
  addition, $H$, $H_0$, and $\P$ are unique, and satisfy the condition
  of Remark~\ref{remark.classification_finite_blocks}.
\end{theorem}

\begin{proof}
  The statement is obvious if the finite blocks are singletons. 
  Otherwise, by
  Lemma~\ref{lemma.sym_on_finite_blocks}, $G$ acts on the set of finite
  blocks as $\P=\sg_\infty$. Use the upcoming
  Proposition~\ref{proposition.tower} to classify the action of $G$ on
  its blocks (the \emph{tower of $G$}) and the upcoming
  Proposition~\ref{proposition.finite_blocks_structure} to lift this
  classification to $G$ itself.
\end{proof}

A positive answer to Macpherson's question follows immediately thanks
to the description of the orbit algebras of the groups
$\groupfiniteblocks[\P]{H_0}{H}$ from
Proposition~\ref{proposition.wreath_direct_orbit_algebra}.

\begin{corollary}[Macpherson on one \blockofblocks]
  \label{corollary.finite_blocks}
  Let $G$ be a closed $P$-oligomorphic permutation group such that
  $\nested{G}$ consists of a single \blockofblocks. Then,
  $\orbitalgebra{G}$ is an invariant ring of permutation group, hence
  finitely generated, Cohen-Macaulay, and of algebraic dimension
  the number of $H$-orbits (of non trivial subsets),
  where $H$ is defined by the classification.
\end{corollary}
\begin{remark}
\label{remark.refined_bound}
  Until now, the lower bound provided by the nested block system evoked in 
  Remark~\ref{remark.lower_bound_from_nested}
  was calculated using Example~\ref{example.wreath_products} when it came to stable 
  \blocksofblocks in $\nested{G}$. 
  With the notations of this section, it was based on the (possibly infinite index) supergroup 
  $H_0\wreath \P$: namely, the provided lower bound for the algebraic dimension 
  was the cardinality of the age of $H_0$.
  Corollary~\ref{corollary.finite_blocks} hands a refinement of this bound, that
  is based on the subgroup $H\wreath \P$ and tight on the relevant 
  restriction of the group.
\end{remark}

In Section~\ref{section.classification_general_case} the strategy
to tackle a group $G$ with several \blocksofblocks will be to consider
the restrictions of $G$ on each of its \blocksofblocks, and patch
together their properties. This will use the following technical
corollary.

\begin{corollary}
  \label{corollary.minimal_subgroup_finite_blocks}
  Let $G$ be a closed $P$-oligomorphic permutation group such that
  $\nested{G}$ consists of a single \blockofblocks; write it as
  $G=\groupfiniteblocks[\P]{H_0}{H}$ using the classification
  of Theorem~\ref{theorem.classification_finite_blocks}, and
  let $\PM$ be the minimal finite index normal subgroup of $\P$.
  Then, any finite index normal subgroup $\tilde{G}$ of $G$ is of the form
  $[\tilde{H_0}, H^\infty, \tilde{M}]$, with
  $H \leq \tilde{H_0} \leq H_0$ and $\PM \leq \tilde{M} \leq \P$. In
  particular, $K=H\wreath \PM$ is the minimal finite index normal
  subgroup of $G$.
\end{corollary}

\begin{proof}
  Since $\tilde{G}$ is of finite index, its nested block system is
  still equal to $\nested{G}$ by
  Lemma~\ref{lemma.finite_index_same_nested}, and its action on the
  maximal finite blocks is a normal subgroup of finite index of $\P$.
  Using the classification of
  Theorem~\ref{theorem.classification_finite_blocks}, $\tilde{G}$ is
  of the form
  $[\tilde{H_0}, \tilde{H}^\infty, \tilde{M}]$, with the expected
  group inclusions: $\tilde{H_0} \triangleleft H_0$,
  $\tilde{H} \triangleleft H$, and $\tilde{M}\triangleleft \P$.
  Lemma~\ref{lemma.subwreath} also states that it contains
  $\tilde{H}\wreath \tilde{M}$ as a finite index normal subgroup,
  while $G$ contains $H\wreath \P$ and thus $H\wreath \tilde{M}$ as
  finite index normal subgroups. Considering
  Lemma~\ref{corollary.subwreath_index}, we need to have
  $\tilde{H} = H$ for $\tilde{G}$ to be of finite index in $G$.
\end{proof}

\subsection{Action on the set of blocks}
\label{section.finite_blocks.action}

The sequel of Section~\ref{section.finite_blocks} is devoted to the
statement and proof of the two propositions used in the proof of
Theorem~\ref{theorem.classification_finite_blocks}. 

From now on, \textbf{we assume that $G$ acts on the set of finite
  blocks as $\P=\sg_\infty$}. The following two technical lemmas
strengthen this assumption by showing that, for an appropriate
enumeration of the elements within in each block, $G$ can permute the
blocks while preserving that enumeration.

\begin{lemma} \label{lemma.sympa} 
  Take any finite collection $(B_{i_1},\ldots,B_{i_k})$ of blocks; then
  $\fix_G (B_{i_1}, \ldots, B_{i_k})$ acts on the remaining blocks as $\sg_\infty$.
\end{lemma}

\begin{proof}
  Take $k$ in $\NN$.
  As $\fix_G (B_1, \ldots, B_k)$ is a normal subgroup of finite index
  of $\stabilizer_G(B_1,\ldots,B_k)$, it acts on the remaining blocks as a 
  subgroup of finite index of $\sg_{\infty}$, which may only be $\sg_\infty$ itself.
  By conjugation of the blocks, the same holds for any collection 
  $(B_{i_1}, \ldots, B_{i_k})$ of blocks.
\end{proof}

\begin{lemma}
\label{lemma.ladder}
  There exists an ordering $b_{1,i},\ldots b_{m,i}$ of the elements
  within each block $B_i$ such that (the closure of) $G$ contains
  $\Id_m \mathbin{\square} \sg_{\infty} = \Id_m \wreath \sg_{\infty}$ as a permutation subgroup.
\end{lemma}

\begin{proof}
  Since $G$ acts
  by $\sg_\infty$ on the blocks, there exists for each $i>1$ a permutation
  $\tau_{1,i}^{(0)} \in G$ that swaps $B_1$ and $B_i$ and
  stabilizes all the other blocks. Take now $k\geq 0$; using
  Lemma~\ref{lemma.sympa} there exists a permutation $\tau_{1,i}^{(k)}$
  that not only swaps $B_1$ and $B_i$, but also fixes all the
  (other) blocks in $B_2,\ldots,B_k$.

  Take an infinite sequence
  $\tau_{1,i}^{(0)},\ldots,\tau_{1,i}^{(k)},\ldots$~. Noting that there are
  only finitely many possibilities for the restriction of $\tau_{1,i}^{(k)}$
  to $B_1\cup B_i$, we can extract a subsequence with
  always the same restriction. Thus, using the property of closure, there exists in 
  $G$ a permutation $\tau_{1,i}$ which swaps $B_1$
  and $B_i$ and fixes all the other blocks. This permutation need not
  be of order 2 though.

  Say that $\tau_{1,i}$ and $\tau_{1,j}$ are equivalent if
  their restrictions to $B_1 \cup B_i$ and $B_1\cup B_j$
  coincide up to renaming the elements of $B_i$ (or $B_j$).

  Now consider the map $i \mapsto \tau_{1,i}$. It takes finitely many
  values, and therefore there exists $i$ and $j$ such that
  $\tau_{1,i}$ and $\tau_{1,j}$ are equivalent. Define
   $\tau'_{i,j} = \tau_{1,i} \tau_{1,j}^{-1} \tau_{1,i}\,$.

  \input{\figures figure_proof_finite_blocks1}
  
  Now check that
  \begin{itemize}
  \item $\tau'_{i,j}$ swaps $B_i$ and $B_j$ ``straightforwardly'': that is
    its restriction to $B_i\cup B_j$ is of order 2 (see
    Figure~\ref{figure.proof_finite_blocks1});
  \item $\tau'_{i,j}$ stabilizes $B_1$;
  \item $\tau'_{i,j}$ fixes all the other blocks (pointwise).
  \end{itemize}

  We may then conjugate $\tau'_{i,j}$ to stabilize some block $B_k$ instead 
  of $B_1$, with $k$ as large as desired, and still swap straightforwardly 
  $B_i$ and $B_j$ while fixing the remaining blocks.

  Therefore there exists in $G$, which we recall is assumed to be closed, 
  a permutation
  $\tau_{i,j}$ of order $2$ that swaps $B_i$ and $B_j$ and fixes all
  the other blocks. By conjugation, we can find for each $n$ a
  similar permutation $\tau_n$ swapping $B_n$ and $B_{n+1}$.

  Choose an arbitrary ordering $b_{1,1},\ldots,b_{m,1}$ of $B_1$.
  Define the ordering $b_{1,2},\ldots,b_{m,2}$ of $B_2$ so that
  $\tau_1$ is the trivial swap, meaning that it swaps $b_{1,r}$ and $b_{2,r}$
  for each $r$. Proceed similarly to order the elements of $B_3$ so
  that $\tau_2$ is the trivial swap, and so on 
  (Figure~\ref{figure.proof_finite_blocks2} shows the stage $k-1$).

  \input{\figures figure_proof_finite_blocks2}  

  Conclusion: the $\tau_n$'s generate $\Id_m \mathbin{\square} \sg_\infty$ as a
  permutation subgroup of $G$, as desired.
\end{proof}

\subsection{Towers and their classification}
\label{section.finite_blocks.tower}

While the previous subsection dealt with the way the finite blocks could
permute, this subsection is going to focus on what can happen within the 
blocks when they do not permute (the results obtained above state that
the actions on and within the blocks can be decorrelated anyway).

\begin{definition}
  Let $S_{\blocksystem} = S_{\blocksystem}^G = \stabilizer_G(\blocksystem)$ be the
  kernel of the morphism that maps $G$ onto its induced action on the set
  of blocks, and, for $i\geq 0$, set
  $H_i = H_i^G = \fix_{S_{\blocksystem}}(B_1, \ldots, B_i)_{|B_{i+1}}$. 
  We call the sequence $H_0, H_1, H_2, \ldots$ the \textbf{tower} of
  $G$ with respect to the block system $\blocksystem$. The groups $H_i$ are
  considered up to a permutation group isomorphism.
\end{definition}

\begin{remark}
  \label{remark.tower_indep_of_order}
\begin{itemize}
  \item By conjugation, using Lemma~\ref{lemma.sym_on_finite_blocks}, 
  the tower does not depend on the ordering $B_1,B_2,\dots$ of the
  blocks. In others words, $H_i$ can be obtained by fixing (pointwise)
  any $i$ blocks and taking the restriction to any other block.
\item Up to permutation group isomorphisms, the sequence $H_0,H_1,\dots$ forms a
  weakly decreasing chain of subgroups of $\sym(\{1, 2, \ldots, m\}$), where
  $m$ is the cardinality of the blocks. Take indeed an arbitrary block
  $B'_0$ and label its elements by $\{1, 2, \ldots, m\}$; then define
  each $H_i$ as $\fix_{S_{\blocksystem}}(B'_1, \ldots, B'_i)_{B'_0}$
  where $B'_1,\dots,B'_i$ are $i$ arbitrary distinct blocks. In
  addition, each $H_{i+1}$ is normal in $H_i$.
\end{itemize}
\end{remark}

The above definition and remark also apply
to a permutation group of a finite set, as long as it acts on the
(finitely many) blocks as the full symmetric group.

\begin{example}[Fundamental examples]
  Let $H$ be a finite permutation group. The tower of
  $H \wr \sg_{\infty}$ (resp. $H \mathbin{\square} \sg_{\infty}$) for its natural
  block system is $H, H, H \cdots$ (resp. $H, \Id, \Id \cdots$).
  The tower of $\groupfiniteblocks{H_0}{H}$ is $H_0, H, H, H \cdots$ .
\end{example}

\begin{center}
\vspace{.5cm}
\begin{tikzcd}[column sep = tiny, row sep = tiny]
  H_0 \wreath \sg_\infty \ar[d, dash] & \longleftrightarrow ~~ H_0,~ H_0,~ H_0,~ H_0 ~ \cdots \\
  \vdots \ar[d, dash] & \\
  \groupfiniteblocks{H_0}{H} \ar[ddd, dash, "< \infty"] 
  & \longleftrightarrow 
  ~~ H_0,~ H_{{\color{white}0}},~ H_{{\color{white}0}},~ H_{{\color{white}0}} ~ \cdots \\
   & \\
   & \\
  H \wreath \sg_\infty \ar[d, dash] 
  & \longleftrightarrow 
  ~~ H_{{\color{white}0}},~ H_{{\color{white}0}},~ H_{{\color{white}0}},~ H_{{\color{white}0}} ~ \cdots \\
  \vdots \ar[d, dash] & \\
  H \mathbin{\square} \sg_\infty & \longleftrightarrow ~~ H,~ \Id,~ \Id,~ \Id ~ \cdots \\
\end{tikzcd}
\end{center}

We aim to prove that these are the only possibilities for a tower (so 
there is actually only one prototype of a tower, since the first two
examples are a specialization of the third one).

\iffalse
\begin{lemma}
  \label{lemma.four_blocks}
  Let $G$ be a finite permutation group endowed with a block system
  composed of four blocks on which $G$ acts by $\sg_4$; 
  denote its tower by $H_0, H_1, H_2, H_3$.
  Assume further that the finite version of Lemma~\ref{lemma.ladder} holds.
  Then we have $H_1 = H_2$.
\end{lemma}
\begin{proof}
An element $s$ of $S_\mathcal{B}$ is determined by its action on each block,
which we write as a quadruple.
Let $g$ be an element of $H_1$. Then $S_\mathcal{B}$ has an element $x$ that may be written 
$(1, g, h, l)$, with $h$ and $l$ also in $H_1$.
Let $\sigma$ be an element of $G$ that permutes ``straightforwardly'' the first 
two blocks and fixes the other two (it exists by hypothesis).
By conjugating $x$ with $\sigma$ in $G$, we get an element $y$ in $S_\mathcal{B}$
that we may write $(g, 1, h, l)$, such that $x^{-1}y = (g, g^{-1}, 1, 1)$.
Hence, using Remark~\ref{remark.tower_indep_of_order}, $g$ is actually in $H_2$.
\end{proof}
\fi

\begin{proposition}
  \label{proposition.tower}
  Let $G$ be a closed $P$-oligomorphic permutation group with $\nested{G}$
  consisting of a single \blockofblocks.
  Then, the tower of $G$ has the form $H_0, H, H, H \cdots$, where
  $H_0$ is a finite permutation group and $H$ is a normal subgroup of $H_0$.
\end{proposition}

\begin{proof}
  
  Consider, for any $i \in \NN$, the restriction $G_i$ of 
  $\text{Fix}_G(\cup_{j<i}B_j)$
  to the four next blocks. The tower of this permutation group (for a natural
  extension of the notion to finite groups of the adequate shape) is
  $H_i, H_{i+1}, H_{i+2}, H_{i+3}$. We aim to show that $H_{i+1} = H_{i+2}$,
  which will conclude the proof.
  
  An element $s$ of the blockwise stabilizer $S_i$ of $G_i$ is 
  determined by its action on each block, which we write as a quadruple.
  Let $g$ be an element of $H_{i+1}$. Then $S_i$ has an element $x$ 
  that may be written $(1, g, h, l)$, with $h$ and $l$ also in $H_{i+1}$.
  Let $\sigma$ be an element of $G_i$ that permutes ``straightforwardly'' the first 
  two blocks and fixes the other two (Lemma~\ref{lemma.ladder} states that
  such an element actually exists).
  By conjugating $x$ with $\sigma$ in $G_i$, we get an element $y$ in $S_i$
  that we may write $(g, 1, h, l)$, so that $x^{-1}y = (g, g^{-1}, 1, 1)$.
  Hence, using Remark~\ref{remark.tower_indep_of_order}, $g$ is actually in $H_{i+2}$.

\end{proof}

\subsection{Lifting of the classification from towers to groups}
\label{section.finite_blocks.lift}

The two results to follow will show that $G$ is uniquely defined by 
its tower, by first recovering the blockwise stabilizer of the group from the
tower, and then using the result of ``straightforward'' permutation we
proved in Lemma~\ref{lemma.ladder}.

\begin{lemma}
\label{lemma.same_tower_same_stab}
The tower of $G$ w.r.t. $\blocksystem$ uniquely determines its blockwise
stabilizer $S_{\blocksystem}$.
\end{lemma}

\begin{proof}
  Let $(H_i)_i$ be the tower of $G$ w.r.t. $\blocksystem$ and
  $S_{\blocksystem}$ be the blockwise stabilizer of $\blocksystem$.

  Using that $S$ is closed, it is sufficient to prove that, for any
  $l\geq 0$, the restriction $S_\ell$ of $S$ to the first $\ell$ blocks is
  determined by the tower; or equivalently to any $\ell$ blocks (recall
  that the order of the blocks is irrelevant). To this end, we will show
  that $S_\ell$ admits an expression that involves only explicit subdirect products 
  and the $H_i$'s (which will do the job thanks to 
  Proposition~\ref{prop.subdirect_characteristic}).

\begin{tikzpicture}[scale=.3]
  \draw[draw=violet, very thick](0,0) ellipse (1.7cm and 4cm);
  \draw[draw=violet, very thick](4,0) ellipse (1.7cm and 4cm);
  \draw[draw=violet, very thick](8,0) ellipse (1.7cm and 4cm);
  \draw[draw=violet, very thick](12,0) ellipse (1.7cm and 4cm);
  \draw[draw=violet, very thick](16,0) ellipse (1.7cm and 4cm);
  \fill[color=violet!10, pattern=north east lines, opacity=.3] (0,0) ellipse (1.7cm and 4cm);
  \fill[color=violet!10, pattern=north east lines, opacity=.3] (4,0) ellipse (1.7cm and 4cm);
  \fill[color=violet!10, pattern=north east lines, opacity=.3] (8,0) ellipse (1.7cm and 4cm);
  \fill[color=violet!10, pattern=north east lines, opacity=.3] (12,0) ellipse (1.7cm and 4cm);
  \fill[color=violet!10, pattern=north east lines, opacity=.3] (16,0) ellipse (1.7cm and 4cm);
  \draw[draw=red, very thick](20,0) ellipse (1.7cm and 4cm);
  \draw[draw=red, very thick](24,0) ellipse (1.7cm and 4cm);
  \draw[draw=red, very thick](28,0) ellipse (1.7cm and 4cm);
  \draw (32,0) ellipse (1.7cm and 4cm);
  \draw (36,0) ellipse (1.7cm and 4cm);
  \draw (40,0) ellipse (1.7cm and 4cm);
  \node (dots) at (45, 0){$\cdots$};
  \draw[>=stealth, <->] (-2,-5) -- (17.5,-5);
  \node (k) at (8,-6){\color{violet} $k$}; 
  \draw[>=stealth, <->] (18.5,-5) -- (30,-5);
  \node (k) at (24,-6){\color{red} $\ell$}; 
  \coordinate (P1) at ($(24.1, 5.5) + (-40:6cm and 0.8cm)$);
  \draw[>=stealth,->] (P1) arc (-40:290:6cm and 0.8cm);
  \node (gp) at (24,7.2){$H_{k,\ell}$};
\end{tikzpicture}

  In order to proceed by induction on $\ell$, we consider the larger family
  $(H_{k,\ell})_{k\geq 0,\ell>0}$, where $H_{k,\ell}$ is the restriction on
  $l$ blocks of the fixator of $k$ other blocks in $S$. Of course, 
  $H_{0,\ell} = S_\ell$.

  First, note that we have $H_{k,1}=H_k$ for all $k$. 
  This gives the base case for the induction. 
  We now take $\ell>1$, and express $H_{k,\ell}$ as a subdirect product involving
  only $H_{k',\ell'}$ with $\ell'<\ell$ (and incidentally also $k'+\ell'\leq k+\ell$).

  Write $l=\ell_1+\ell_2$ with $\ell_1>0$ and $\ell_2>0$, partition the $\ell$ blocks
  into $\ell_1$ and $\ell_2$ blocks, and let $E_1$ and $E_2$ be their respective
  union. Considering the action of $H_{k,\ell}$ on $E_1$ and $E_2$ provides
  the desired expression:
  \begin{displaymath}
    H_{k,\ell} = \operatorname{Subdirect}((G_1, G_2), (N_1,N_2))\,,
  \end{displaymath}
  where:
  \begin{alignat*}{2}
    G_1 &= \restrict{H_{k,\ell}}{E_1}              &&= H_{k,\ell_1}      \\
    G_2 &= \restrict{H_{k,\ell}}{E_2}              &&= H_{k,\ell_2}      \\
    N_1 &= \fix_{H_{k,\ell}}(E_2)_{|E_1} &&= H_{k+\ell_1, \ell_2}  \\
    N_2 &= \fix_{H_{k,\ell}}(E_1)_{|E_2} &&= H_{k+\ell_2, \ell_1}   \qedhere
  \end{alignat*}
\end{proof}

\begin{remark}
  If desired, more explicit formulae can be obtained, by imposing the partition.
  For instance, even splittings have the pleasant property that $G_1=G_2$ and 
  $N_1 = N_2$, which allows to illustrate the process by a binary tree:
  following is for example a recursion tree to express $H_{0,8}$ as a subdirect product
  of $H_i$'s, assuming that the left (resp. right) hand child of a group is the $G_i$ 
  (resp. $N_i$) of the subdirect product making this group (so a group is
  determined by its two children). 
  This recursion tree generalizes immediately to any
  $H_{0,2^n}$, which is sufficient to retrieve $S$ by closure.
  \begin{tikzpicture}[level distance=1.5cm,
    level 1/.style={sibling distance=7cm},
    level 2/.style={sibling distance=3.5cm}
    , level 3/.style={sibling distance=2cm}
    ]
    \node {$H_{0,8}$}
    child {node {$H_{0,4}$}
      child {node {$H_{0,2}$}
      	child {node {$H_{0,1}=H_0$}}
      	child {node {$H_1$}}
      }
      child {node {$H_{2,2}$}
      	child {node {$H_2$}}
      	child {node {$H_3$}}}
    }
    child {node {$H_{4,4}$}
      child {node {$H_{4,2}$}
      	child {node {$H_4$}}
      	child {node {$H_5$}}}
      child {node {$H_{6,2}$}
      	child {node {$H_6$}}
      	child {node {$H_{7,1}=H_7$}}}
    };
  \end{tikzpicture}
\end{remark}

\begin{proposition}
  \label{proposition.finite_blocks_structure}
  The permutation group $G$ is the natural semidirect product of its
  blockwise stabilizer $S_{\blocksystem}$ and
  $L=\Id_m \mathbin{\square} \sg_\infty$. In particular, it is uniquely defined
  by its tower w.r.t. $\blocksystem$.
\end{proposition}

\begin{proof}
  Use first Lemma~\ref{lemma.ladder} to state that $G$ contains
  $L=\Id_m \mathbin{\square} \sg_\infty$ (for ``ladder'') as a permutation subgroup.
  
  Take $k>0$, the stabilizer $\stabilizer_{\tilde{G}} (B_1,\ldots,B_k)$ of the
  first $k$ blocks is isomorphic to $\stabilizer_{G} (B_1,\ldots,B_k)$
  by Lemma~\ref{lemma.same_tower_same_stab}.

  Now, the group generated by $L_{|B_1 \cup \cdots \cup B_k}$
  and $\stabilizer_G (B_1,\ldots,B_k)_{|B_1 \cup \cdots \cup B_k}$ is 
  a subgroup of the restriction of $G$ (actually of
  $\stabilizer_G (B_1 \cup \cdots \cup B_k)$)
  to the same domain. Moreover, the latter is of size 
  $|\sg_k|.|\stabilizer_G (B_1,\ldots,B_k)|$ 
  (consider the morphism that projects onto the action on the blocks);
  therefore the two groups are equal.

  Finally, since $G$ acts on the blocks as the symmetric group, which is the
  closure of the group of all finitely supported permutations, each of its 
  elements is also the simple limit of a sequence of finitely supported permutations.
  Since we just showed that the restrictions to any finite number of blocks
  are uniquely defined by the tower, so is the whole (closed) group $G$.
\end{proof}

This provides the final piece of the proof of the classification
theorem in the case of a single superblock, and we may now move on to
the general case.

\iffalse

\begin{corollary}
\label{corollary.subgroups_of_one_superblock}
  Let $\tilde{G}$ be a finite index normal subgroup of $[H_0, H ^\infty]$. 
  Then $\tilde{G}$ is some $[\tilde{H_0}, H^\infty]$, with $H \leq \tilde{H_0} \leq H_0$.
\end{corollary}

\begin{proof}[Sketch of proof]
  Since $\tilde{G}$ is of finite index, its nested block system is still equal
  to $\nested{G}$ by Lemma~\ref{lemma.finite_index_same_nested}, and it is
  still $P$-oligomorphic so the action on the maximal finite blocks is still
  isomorphic to $\sg_\infty$. Using the classification result of this section,
  $\tilde{G}$ is thus some $[\tilde{H_0}, \tilde{H}^\infty]$, with the expected
  group inclusions: $\tilde{H_0} \leq H_0$ and $\tilde{H} \leq H$.
  The classification also states that it contains $\tilde{H}\wreath \sg_\infty$
  as a finite index subgroup, so considering Lemma~\ref{lemma.subwreath}
  we need to have $\tilde{H} = H$ for $\tilde{G}$ to be of finite index in $G$.
\end{proof}
\fi

\section{Classification of (closed) $P$-oligomorphic groups}
\label{section.decoupage_recollage}
\label{section.classification_general_case}

Let $G$ be a closed $P$-oligomorphic group. In
Subsection~\ref{minimal_finite_subgroup} we exploit the results from 
Subsection~\ref{section.canonical} on the blocks systems of $G$ and the
classification of closed $P$-oligomorphic groups with one single
\blockofblocks of Section~\ref{section.finite_blocks} to give a
constructive description of the minimal finite index subgroup $K$ of
$G$. This subgroup is the first piece of the classification of $G$.

In Subsection~\ref{subsection.diagonal}, we explicit $G$ as a
semidirect product of $K$ and a finite permutation group acting
diagonally on its nested block system.

Finally, in Subsection~\ref{subsection.classification}, we classify 
(closed) $P$-oligomorphic groups, using a finite data structure involving
some finite group and a choice of decorated block system for it. We
first handle the case when $G$ does not act as $\Rev(\QQ)$ or $\Rev(\QQ/\ZZ)$
on any of its \blockofblocks, and then show how to be inclusive.

\subsection{The minimal finite index subgroup}
\label{minimal_finite_subgroup}

\begin{center}
\begin{tikzpicture}[scale=.5]
\foreach \a in {0,1,...,10}{
  \draw (2*\a,12) node[scale=.5,circle, fill=black!60]{} ;
  \draw (2*\a,11) node[scale=.5,circle, fill=black!60]{} ;
  \draw (2*\a,10) node[scale=.5,circle, fill=black!60]{} ;
  \draw (2*\a,8) node[scale=.5,circle, fill=black!60]{} ;
  \draw (2*\a,7) node[scale=.5,circle, fill=black!60]{} ;
  \draw (2*\a,5) node[scale=.5,circle, fill=black!60]{} ;
  \draw (2*\a,4) node[scale=.5,circle, fill=black!60]{} ;
  \draw (2*\a,2) node[scale=.5,circle, fill=black!60]{} ;
  \draw (2*\a,0) node[scale=.5,circle, fill=black!60]{} ;
	\draw[draw=blue!80, thick] (2*\a+.4,10) -- (2*\a+.4,12);
	\draw[draw=blue!80, thick] (2*\a+.4,12) arc (0:180:.4);
	\draw[draw=blue!80, thick] (2*\a-.4,12) -- (2*\a-.4,10);
	\draw[draw=blue!80, thick] (2*\a-.4,10) arc (180:360:.4);
	
	\draw[draw=blue!80, thick] (2*\a+.4,7) -- (2*\a+.4,8);
	\draw[draw=blue!80, thick] (2*\a+.4,8) arc (0:180:.4);
	\draw[draw=blue!80, thick] (2*\a-.4,8) -- (2*\a-.4,7);
	\draw[draw=blue!80, thick] (2*\a-.4,7) arc (180:360:.4);
	
	\draw[draw=blue!80, thick] (2*\a+.4,4) -- (2*\a+.4,5);
	\draw[draw=blue!80, thick] (2*\a+.4,5) arc (0:180:.4);
	\draw[draw=blue!80, thick] (2*\a-.4,5) -- (2*\a-.4,4);
	\draw[draw=blue!80, thick] (2*\a-.4,4) arc (180:360:.4);
  }
	\node (points1) at (23.2,11){\ldots};
	\node (points2) at (23.2,7.5){\ldots};
	\node (points3) at (23.2,4.5){\ldots};
	\node (points4) at (23.2,2){\ldots};
	\node (points5) at (23.2,0){\ldots};
	\node (pointsv) at (11,14){\vdots};
	\node (pointsv2) at (11,-1.7){\vdots};

	\superblock{0}{26}{10}{12}{.7}{purple!70}
	\superblock{0}{24.3}{7}{8}{.7}{purple!70}
	\superblock{0}{26}{4}{5}{.7}{purple!70}
	\superblock{0}{26}{2}{2}{.7}{purple!70}
	\superblock{0}{26}{0}{0}{.7}{purple!70}
  \draw (26,8) node[scale=.5,circle, fill=black!60]{} ;
  \draw (26,7) node[scale=.5,circle, fill=black!60]{} ;
	\draw[draw=blue!80, thick] (26.4,7) -- (26.4,8) arc (0:180:.4)
    (25.6,8) -- (25.6,7) arc (180:360:.4);
\label{figure.naked_nested}
\end{tikzpicture}
\end{center}

Following Definition~\ref{theorem.canonical_block_system}, let
$\nested{G}$ be the nested block system of $G$; recall that it
consists in a partition of the set of maximal finite blocks (but the
kernel of $G$) into finitely many \blocksofblocks $(\BB)_j$. Let
$\stab{G}{\nested{G}}$ be the stabilizer of the \blocksofblocks; it is a
finite index normal subgroup of $G$ and therefore, by
Lemma~\ref{lemma.finite_index_same_nested}, has the same \blocksofblocks.

We now consider the restriction
$G\BBindex=\restrict{\stab{G}{\nest}}{E\BBindex}$ of
$\stab{G}{\nested{G}}$ on the support $E\BBindex$ of each
\blockofblocks $\BB$. It admits a single \blockofblocks, so that we
can use the classification result of
Theorem~\ref{theorem.classification_finite_blocks}; define
accordingly $H_0\BBindex$, $H\BBindex$, and $\P\BBindex$, such that
$G\BBindex$ is isomorphic to
$\groupfiniteblocks[\P\BBindex]{H_0\BBindex}{{H\BBindex}}$. Following
Corollary~\ref{corollary.minimal_subgroup_finite_blocks}, let
$K\BBindex=H\BBindex\wreath\PM\BBindex$ be the minimal finite index
normal subgroup of $G\BBindex$.

Recall
that, by Lemma~\ref{lemma.subwreath},
$K\BBindex=H\BBindex\wreath\PM\BBindex$ is the minimal finite index
normal subgroup of $G\BBindex$, that one has
$G\BBindex = K\BBindex H_0\BBindex$ and that $H_0\BBindex$ acts
diagonally on $E\BBindex$.

To be concrete, use Lemma~\ref{lemma.ladder} to choose a \emph{coherent
enumeration} of the elements of each block $B\BBindex_i$ of
$\BB$: for each $i,i'$ there exists $g\in G$ that maps
$B\BBindex_i$ to $B\BBindex_{i'}$ while preserving the enumeration. From
now on, we use this chosen enumeration to implicitly identify
elements of $B\BBindex_i$ and of $B\BBindex_{i'}$ when meaningful.
Recall that $\PM\BBindex$ is obtained by considering the
homomorphic image $\P\BBindex$ of $G\BBindex$ acting on the blocks in $\BB$
and, if needed, taking its normal subgroup of index $2$ to ensure that it
contains no proper finite index normal subgroup.
$H\BBindex$ can be obtained by picking arbitrarily two blocks
$B\BBindex_0$ and $B\BBindex_1$ in $\BB$, and taking the restriction
to $B\BBindex_1$ of the subgroup of $G$ that fixes $B\BBindex_0$ and
stabilizes $B\BBindex_1$. Recall also that $\PM\BBindex$ is $\sg_\infty$
whenever the blocks are non trivial; otherwise, $H\BBindex$ is the
trivial permutation group on one element, and $K\BBindex=\PM\BBindex$
is one of $\sg_\infty$, $\Aut(\QQ)$ or $\Aut(\QQ/\ZZ)$.

In addition, set by convention $B\BBindex[0]_0=\ker{G}$,
$\BB[0]=\{\ker G\}$, and $K\BBindex[0]=\Id_{\ker{G}}$.

\begin{remark}
  The groups $K\BBindex$ and $K\BBindex[j']$ are conjugate whenever the
  \blocksofblocks $\BB$ and $\BB[j']$ are in the same $G$-orbit.
\end{remark}

\begin{proposition}
  \label{proposition.K_minimum}
  Let $G$ be a closed $P$-oligomorphic group. Then
  $K=\prod_j K\BBindex$ is the minimal finite index normal subgroup of $G$.
\end{proposition}

Let us start by proving the following result.

\begin{lemma}
\label{lemma.K_minimal}
A permutation group $K$ of the above form admits no (proper) finite index 
normal subgroup.
\end{lemma}

\begin{proof}
  If a subgroup $\tilde{K}$ of $K$ is normal and of 
  finite index, Lemma~\ref{lemma.finite_index_same_nested} states that
  $K$'s \blocksofblocks are still \blocksofblocks for $\tilde{K}$; and since they are
  stable under $K$, they are also stable under any subgroup of $K$. 
  Using the classification
  of Corollary~\ref{corollary.minimal_subgroup_finite_blocks}, the restrictions 
  of $K$ to any 
  \blocksofblocks have no finite index normal subgroup, so $\tilde{K}$ has the same
  restrictions and is thus a subdirect product of these.
  There remains to show that there is no synchronization between these parts.
  Again, the $K\BBindex$ have no finite index normal subgroups, so no finite 
  synchronization (that would be linked to a proper normal subgroup of finite index)
  is to consider; and the case of infinite synchronizations is excluded by
  Lemma~\ref{lemma.stable_superblocks_are_indep}.
\end{proof}

\begin{proof}[Proof of Proposition~\ref{proposition.K_minimum}]
  We are going to reduce $G$ down to $\fisubgroup$ by applying two 
  successive reductions to a normal subgroup of finite index, which 
  will conclude the proof using Lemma~\ref{lemma.K_minimal}.

  Recall that the intersection of two finite index normal subgroups is
  again a finite index normal subgroup. Hence the finite index normal
  subgroups of $G$ form a lattice. It is not guaranteed \emph{a priori} to
  have a minimal element though.

  We consider the nested block system $\nested{G}$ introduced in
  Section~\ref{section.canonical}. By
  Lemma~\ref{lemma.finite_index_same_nested}, at each step, the nested block
  system of the finite index normal subgroup will still be $\nest=\nested{G}$. 
  In particular, the kernel will not grow bigger.
  Denote as earlier %
  by $(BB\BBindex)_j$ the \blocksofblocks of $\nested{G}$
  and by $(E\BBindex)_j$ their respective supports.

  Let $\stab{G}{\nest}$ be the finite index normal subgroup of $G$ 
  that stabilizes the \blocksofblocks in $\nested{G}$, which is the first reduction.
  
  Assume first $j\ne 0$. We may apply the
  classification result of Lemma~\ref{lemma.subwreath}~:
  $\restrict{\stab{G}{\nest}}{E\BBindex}$ contains as a finite index
  normal subgroup some wreath product
  $H\BBindex \wreath \PM \BBindex$, where $\PM \BBindex$ is
  given by the action of $\stab{G}{\nest}$ on the set of maximal finite blocks
  of $E\BBindex$ (possibly up to taking an index 2 normal subgroup),
  while $H\BBindex$, which acts within the blocks, is given by $H$ in the
  tower $H_0,H,H,H,\ldots$ of $\restrict{\stab{G}{\nest}}{E\BBindex}$. 
  This subgroup is isomorphic to $K\BBindex$, which also implies that it 
  contains no proper normal subgroup of finite index; therefore, thanks to
  the aforementionned lattice structure, it is the minimal finite index 
  normal subgroup of $\restrict{\stab{G}{\nest}}{E\BBindex}$.

  The same conclusions can be
  reached trivially for $j=0$ (recall that $E\BBindex[0]$ is the
  kernel of $G$ and $K\BBindex[0]$ is the trivial group thereupon).

  Now is the time for the second reduction.
  Consider the finitely many cosets of $K\BBindex$ in
  $\restrict{\stab{G}{\nest}}{E\BBindex[j]}$; the latter (and therefore
  $\stab{G}{\nest}$) acts by permutation on these cosets. Now denote
  by $\tilde{\fisubgroup}$ the kernel of its simultaneous action
  on the whole set of cosets for all $j$.
  At this point, for each given $j$, the restriction of $\tilde{\fisubgroup}$ 
  to $E\BBindex$ is a subgroup of $K\BBindex$; 
  it could be a proper subgroup at first glance,
  due to the constraints inherited from the action on the other sets of 
  cosets (the cosets from other $E\BBindex$, for different $j$).
  However, thanks to the minimality of $\fisubgroup\BBindex$,
  we do have $\restrict{\tilde{\fisubgroup}}{E\BBindex}=\fisubgroup\BBindex$, 
  for every $j$~: $\tilde{\fisubgroup}$ is a subdirect product of 
  the direct product $\fisubgroup$.
  Conclude the proof by replaying that of Lemma~\ref{lemma.K_minimal} to show that 
  $\tilde{\fisubgroup}$ is actually the whole direct product, namely $\fisubgroup$ 
  itself, and use Lemma~\ref{lemma.K_minimal} and the lattice structure to state
  its status of minimum.
\end{proof}

\begin{remark}
  From Remark~\ref{remark.independent_parts},
  $\orbitalgebra{\fisubgroup}$ is a free algebra, possibly tensored
  with some finite dimensional diagonal algebra, which is finitely
  generated. Explicitely, we may write:
  \[\orbitalgebra{\fisubgroup} = \bigotimes_i \orbitalgebra{\fisubgroup_i} 
  \otimes \orbitalgebra{\ker{\fisubgroup}} = 
  \mathcal{A}_f \otimes \orbitalgebra{\ker{\fisubgroup}}
  = \bigoplus_k e_k~ \mathcal{A}_f\]
  the $e_k$ being the subsets of the kernel of $\fisubgroup$;
  in other words, $\orbitalgebra{\fisubgroup}$ is a Cohen-Macaulay 
  algebra over the free subalgebra $\mathcal{A}_f$.
\end{remark}

\begin{corollary}
  The lower bound provided by the nested block system according to 
  Remark~\ref{remark.refined_bound} is tight.
\end{corollary}

\begin{proof}
The algebraic dimension of the algebra of $K$ is the sum of dimensions
of the $\mathcal{A}_f$'s, which coincide with the lower bound handed by the nested
system; and the algebraic dimension of $\orbitalgebra{G}$ is the same by finiteness
of the index of $\fisubgroup$.
\end{proof}

\subsection{Semidirect product structure and diagonal action}
\label{subsection.diagonal}

In this section, we generalize the notion of diagonal action of
Section~\ref{section.finite_blocks}, and prove that $G$ is a product
$FK$, where $F$ is a finite permutation group acting diagonally on
$\nested{G}$, or, equivalently, that $G$ is a semidirect product of
$K$ with some natural quotient $F/\finiteK$.

Informally, a permutation acts diagonally if it acts consistently within each block
of a given superblock, and possibly permutes the superblocks but not
the blocks within each superblock. Considering Figure~\ref{figure.naked_nested}, one
may intuitively picture it as a purely vertical action, that
permutes the rows formed by the elements of $\domain$ on the figure
--- and preserves the superblock structure of course.
There remains to define formally the vertical alignment.

An indexing of the $B\BBindex$ within each \blockofblocks $\BB$ is
\emph{coherent} if, for any two \blockofblocks $\BB$ and $\BB[j']$ in
the same orbit, there exists some permutation $g_{j, j'}$ in $G$ which
maps each block $B\BBindex_{i}$ to $B\BBindex[j']_{i}$.

\begin{lemma}
  A $P$-oligomorphic group always admits a coherent indexing of the
  blocks within each \blockofblocks.
\end{lemma}
\begin{proof}
  We may, for example, proceed as follow: for each orbit of
  \blockofblocks, pick one \blockofblocks $\BB[j_0]$ and choose any
  indexing for this \blockofblocks; then, for each \blockofblocks
  $\BB[j]$ in its orbit, pick some permutation $g_{j_0,j}\in G$
  mapping $\BB[j_0]$ to $\BB$; use $g_{j_0,j}\in G$ to transport the
  indexing of the blocks in $\BB[j_0]$ to $\BB[j]$; finally, for each
  $\BB[j]$ and $\BB[j']$ in the orbit of $\BB[j_0]$, define
  $g_{j,j'}=g_{j_0,j'}g_{j_0,j}^{-1}$ (here and elsewhere in the
  paper, composition is denoted from right to left).
\end{proof}

From now on, we fix a coherent indexing of the blocks within the
\blockofblocks. We further assume without loss of generality that the
chosen coherent enumeration of the elements within the blocks is
preserved by the $g_{j,j'}$'s: the enumeration of the elements of a
block $B\BBindex_i$ is mapped by $g_{j,j'}$ to the enumeration of the
elements of $B\BBindex[j']_i$. To achieve this, we may proceed as
above: choose the coherent enumerations in the superblock $\BB[j_0]$,
and then use the $g_{j_0, j}$ to transport them to the other
superblocks in the same orbit.

It will be convenient to index globally the elements of $E$. To this
hand, choose one block $U\BBindex$ within each superblock in such a
way that, whenever $U\BBindex$ and $U\BBindex[j']$ are in the same
orbit, they have the same index in $\BB$ and $\BB[j']$ respectively.
Let $U=\sqcup_j U\BBindex$. Now any element $e$ of $E$ can be uniquely
described by two coordinates: ``horizontally'' the index $i$ of the
block $B\BBindex$ containing it; ``vertically'', the element $u\in U$
which corresponds to $e$ in the coherent enumeration of $U\BBindex$.
For convenience, we identify $e$ with the pair $(u,i)$.

\begin{definition}
  A permutation $g\in G$ \emph{acts diagonally} on $E$ if 
  any pair of elements $((u,i_1),\,(u,i_2))$ is mapped to some
  pair $((u',i_1),\,(u',i_2))$.

  Let respectively $\finiteG$ (resp. $\finiteK$) be the collection of
  the permutations of $G$ (resp. $K$) that acts diagonally on $E$.
\end{definition}

\begin{remark}
  The collections $\finiteG$ and $\finiteK$ are groups and can be 
  canonically identified with finite permutation groups of $U$. 
  In addition,
  $\finiteK$ is the direct product of each $H\BBindex$ acting on
  $U\BBindex$. Finally, $\finiteG$ as a finite permutation group does
  not depend on the choice of the coherent indexing of the blocks
  within the superblocks and, up to isomorphism, does not depend on
  the choice of $U$ or of the coherent enumeration of the elements
  within the blocks.
\end{remark}

We may now state the main result of this section. Here we settle for
$\Rev$-free groups as this is sufficient for this paper. The extension
to the general case is straightforward at the expense of heavier
notations.
\begin{proposition}
  \label{proposition.semidirect_product}
  Let $G$ be a \Rev-free $P$-oligomorphic group, and $K$ and
  $\finiteG$ be defined as above. Then,
  \begin{enumerate}[(i)]
  \item $G=K\finiteG$;
  \item $G/K$ is isomorphic to $\finiteG/\finiteK$;
  \item $G$ is a semi-direct product of $\finiteG/\finiteK$ and $K$.
  \end{enumerate}
\end{proposition}
\begin{proof}
  (i) We take $g\in G$ and aim to write it as a product in
  $K\finiteG$.

  Consider some superblock $\BB$. Let $j'$ be such that $g$ maps
  $\BB[j']$ to $\BB$, and set $h=g g_{j, j'}$; by construction, $h$
  stabilizes $\BB$ and its restriction $\restrict{h}{E\BBindex}$
  is thus in $G\BBindex$. Since $G$ is \Rev-free, we have $\PM\BBindex=\P\BBindex$,
  and $G\BBindex=K\BBindex H_0\BBindex$. Pick accordingly
  $k\BBindex$ in $K\BBindex$ such that
  $k\BBindex \restrict{h}{E\BBindex}$ is in $H_0\BBindex$ and thus
  acts diagonally on $E\BBindex$. Then,
  $k\BBindex \restrict{g}{E\BBindex[j']}$ maps $\BB[j']$ to $\BB$,
  preserving the coherent indexing of blocks and enumeration within
  the blocks.

  Define $k=\prod k\BBindex\in K$ by extending each $k\BBindex$ by the
  identity on the other superblocks. Then, $kg$ acts diagonally, as
  desired.

  (ii) Consider now the morphism $\phi: \finiteG\mapsto G/K$ obtained
  through the embedding of $\finiteG$ in $G$ and the canonical
  projection of the latter on $G/K$. Thanks to (i), $\phi$ is
  surjective. In addition $\ker \phi= \finiteG \cap K=\finiteK$.
  Therefore, $\phi$ is the desired isomorphism between
  $\finiteG/\finiteK$ and $G/K$.

  (iii) is a reformulation of (ii).
\end{proof}

\subsection{Classification of closed $P$-oligomorphic groups}
\label{subsection.classification}

We now have all the ingredients to classify (closed) $P$-oligomorphic groups.
We first use the previous sections to extract from a
$P$-oligomorphic group $G$ a finite piece of information $\Data(G)$.
It consists of a finite permutation group, endowed with a block system
where each block is decorated with a permutation group of its elements
and one of the five highly homogeneous permutation groups(or the
trivial group).
Conversely, we show that, starting from such a permutation group with
decorated blocks $\data$, one can construct an oligomorphic permutation group
$\DataInv(\data)$.

We check that, up to a natural isomorphism, $G$ can
be reconstructed from $\Data(G)$ by $\DataInv$. More generally, we
check that, up to an isomorphism, $\Data$ and $\DataInv$ give a
one-to-one correspondence between finite permutation groups with
decorated blocks and $P$-oligomorphic permutation groups. This
concludes the classification.

\subsubsection{Classification of $\Rev$-free closed $P$-oligomorphic groups}

For the sake of simplicity of exposition, we first tackle the subclass
of \emph{\Rev-free groups}; that is groups that do not act by
$\Rev(\QQ)$ or $\Rev(\QQ/\ZZ)$ on any of their \blocksofblocks. This
is actually sufficient to classify ages of $P$-oligomorphic groups. In
the following section, we detail how $\Data(G)$ can be extended to
also preserve this piece of information.

Let $G$ be a closed $P$-oligomorphic group.
Take again the notations introduced at the beginning of the
previous subsection: $\nested{G} = \{BB\BBindex\}_j$, $K = \prod_j K\BBindex$,
with $K\BBindex = H\BBindex \wreath \Prim$,
and finally the finite blocks $(B\BBindex_0)_j$ arbitrarily picked 
in each \blockofblocks.

Consider the group $G_{<\infty}$ together with its normal subgroup
$K_{<\infty}=\prod_j H\BBindex$, as defined in
Subsection~\ref{subsection.diagonal}, and identify them canonically
with permutation groups of the finite set $\sqcup_j B\BBindex_0$.
\begin{definition}
  Define $\Data(G) = (G_{<\infty}, (B\BBindex_0)_j, (H\BBindex)_j, (\Prim)_j)$.
\end{definition}

\begin{definition}
  A \textbf{permutation group with decorated blocks}
  $(F, B, (H\BBindex)_j, (\Prim)_j)$ consists of a finite
  permutation group $F$ endowed with a block system
  $B = \{B\BBindex\}_j$ together with the choice, for each block
  $B\BBindex$, of
  \begin{itemize}
  \item a normal subgroup $H\BBindex$ of the restriction
    $\restrict{\text{Fix}_F(\cup_{i\neq j} B\BBindex[i])}{B\BBindex}$
    of the pointwise stabilizer of the other blocks,
  \item $\Prim$, one of the three minimal highly homogeneous groups, or
    the trivial group,
  \end{itemize}
  satisfying the following constraints:
  \begin{itemize}
  \item the choices must be the same for $i\neq j$ whenever
    $B\BBindex[i]$ and $B\BBindex$ are in the same $F$-orbit,
  \item $\Prim$ is $\sg_\infty$ whenever $B\BBindex$ is not of size $1$,
  \item at most one $\Prim$ is trivial, and when it is
    $K\BBindex=H\BBindex$ is trivial too.
  \end{itemize}
\end{definition}

\begin{remark}
  Let $G$ be a closed $P$-oligomorphic permutation group. Then $\Data(G)$ is
  a permutation group with decorated blocks (recall the construction of $K$ 
  at the beginning of Subsection~\ref{minimal_finite_subgroup} and its 
  normality in $G$ stated by Proposition~\ref{proposition.K_minimum}).
\end{remark}

\begin{definition}
  \label{definition.Group}
  Let $\data = (F, B, (H\BBindex)_j, (\Prim)_j)$ be a permutation group with
  decorated blocks, and $\domain$ the disjoint union
  $\sqcup_j \domain\BBindex$, where $\domain\BBindex$ is the cartesian
  product of $B\BBindex$ and the domain of $\Prim$.
  For each $j$, take the wreath product
  $K\BBindex=H\BBindex \wr \Prim$ acting naturally on $\domain\BBindex$.
  Finally, let $K$ be the direct product $\prod_j K\BBindex$, acting on $\domain$.

  We define $\DataInv(\data)$ as the smallest permutation group on
  $\domain$ containing both $K$ and $F$ acting diagonally on 
  $\sqcup_j \domain\BBindex$. Denote additionally $H = \prod_j H\BBindex
  = K_{<\infty}$.
\end{definition}

\begin{proposition}
  Let $\data$ be a permutation group with decorated blocks. Define
  $G=\DataInv(\data)$, and use the notations above. Then,
  $G$ is a $P$-oligomorphic permutation group.
\end{proposition}
\begin{proof}
  The subgroup $K$ of $G$ is the direct product of the wreath products
  $H\BBindex \wr \Prim$, and therefore $P$-oligomorphic. %
  This implies the result by Lemma~\ref{lemma.subgroup.rate}.

\end{proof}

We proceed by defining the notion of isomorphism for groups with decorated blocks,
and checking that it matches the classical notion of isomorphism for 
$P$-oligomorphic groups.
\begin{definition}
  Let $\data$ and $\data'$ be two permutation groups with decorated
  blocks. Then, $\data$ and $\data'$ are isomorphic if there exists an
  isomorphism between the underlying groups $F$ and $F'$ that
  transports the block system $B$ and the groups $H\BBindex$ and
  $\PM\BBindex$ to their equivalents in $\data'$.
\end{definition}

\begin{lemma}
  \label{classification.reciprocal_correspondence1}
  Let $\data$ be a permutation group with decorated blocks. Then 
  $\DataInv(\data)$ is the natural semidirect product $K \rtimes F/H$
  with the notations of Definition~\ref{definition.Group}, and
  $\data'=\Data(\DataInv(\data))$ is isomorphic to $\data$.
\end{lemma}

\begin{proof}
  Denote $G=\DataInv(\data)$.
  The groups $K$, $F$ and $K_{<\infty}$ obtained in 
  Subsection~\ref{subsection.diagonal} from $G$
  correspond to the groups of same name from Definition~\ref{definition.Group},
  so the results from that
  subsection apply and the first part is immediate.
  In particular, $K$ is a normal subgroup of
  $G$ of finite index, which in addition, by
  Proposition~\ref{proposition.K_minimum}, admits no finite index
  normal subgroup and is therefore the unique minimal finite index normal
  subgroup of $G$. This in turn implies the uniqueness of
  all the other pieces of $\Data(G)$: the nested block
  system of $G$ is given by
  $\nested{G}=(BB\BBindex)_j$, with $BB\BBindex=(B\BBindex_i)_i$ and, for $i$ in the
  support of $\PM\BBindex$, $B\BBindex_i=B\BBindex \times \{i\}$;
  the permutation subgroups $G_{<\infty}$ and $K_{<\infty}$
  induced respectively by $G$ and $K$ on $\sqcup_j B\BBindex_0$ are
  respectively trivially isomorphic to $F$ and $H$.
\end{proof}

\begin{lemma}
  \label{classification.reciprocal_correspondence2}
  Let $G$ be a \Rev-free $P$-oligomorphic group. Then,
  $G'=\DataInv(\Data(G))$ is isomorphic to $G$.
\end{lemma}
\begin{proof}
  We use the coherent enumeration to identify the elements of each
  $B\BBindex_i$ with that of $B\BBindex \times \{i\}$. Through this
  identification and by construction, we have $K'=K$; since the finite
  groups acting diagonally are the same as well, and using
  Proposition~\ref{proposition.semidirect_product}, we have indeed
  $G' = G$.
\end{proof}

\begin{theorem}
  \label{theorem.revfree_classification}
  \Rev-free $P$-oligomorphic permutation groups are classified by finite
  permutation groups with decorated blocks through the $\Data$ and
  $\DataInv$ reciprocal correspondences.
\end{theorem}
\begin{proof}
  Lemma~\ref{classification.reciprocal_correspondence1} together with
  Lemma~\ref{classification.reciprocal_correspondence1} asserts that
  $\Data$ and $\DataInv$ are reciprocal correspondences, as desired.
\end{proof}

\subsubsection{Extending the classification to all $P$-oligomorphic groups}

We start with an example illustrating that the straightforward
extension of $\Data$ to all $P$-oligomorphic groups does not give a
proper correspondence.
\begin{example}
  Consider the $P$-oligomorphic group $G=\Rev{\QQ} \times \Rev{\QQ}$,
  and the index $2$ subgroup generated by
  $G'=\Aut{\QQ} \times \Aut{\QQ}$ on the one hand and the reversal acting
  simultaneously on the two copies of $\QQ$ on the other hand.

  Let us try to define $\Data$ on $G$ and $G'$ as before; in both
  cases, we get the same data:
  \begin{displaymath}
    \left(\id(\{1,2\}),
      (\{i\})_{i=1,2},
      (\id(\{i\}))_{i=1,2},
      (\Rev(\QQ))_{i=1,2}\right)\,.
  \end{displaymath}
  The information about the synchronization of the reversal on the two
  \blockofblocks is lost.
\end{example}

\newcommand{\E}{\overline}

We now tweak the definition of $\Data$ to keep track of reversals in
the finite group $G_{<\infty}$. To achieve this, each copy of $R=\QQ$
(or of $R=\QQ / \ZZ$) where a reversal can occur will be compressed into
a block of two points instead of a single one.

Let $BB\BBindex$ be a \blockofblocks; if its blocks are of size $1$ and
$G$ acts on them by %
$\PM\BBindex=\Rev(R)$, then define $\E B\BBindex_0$ by
choosing any two points of $\domain\BBindex$; note that
$\E B\BBindex_0$ is not a block anymore, but this is fine. Otherwise,
define $\E B\BBindex_0$ as $B\BBindex_0$.

Define $\E G_{<\infty}$ as before, but using $\sqcup_j \E B\BBindex_0$
instead of $\sqcup_j B\BBindex_0$.

\begin{example}
  With $G$ and $G'$ as in the previous example, $\E G_{<\infty}$ and
  $\E G'_{<\infty}$ both act on $\{1,2\} \sqcup \{1,2\}$. However
  $\E G_{<\infty}$ is of size 4, permuting independently the two
  blocks, whereas $\E G'_{<\infty}$ is of size 2, permuting
  simultaneously the two blocks.
\end{example}

\begin{definition}
  Define
  $\E \Data{(G)} = (\E G_{<\infty}, (\E B\BBindex_0)_j, (H\BBindex)_j,
  (\Prim)_j)$.
\end{definition}

The definition of permutation group with decorated blocks must be
extended accordingly: each $\PM\BBindex$ can now be any one of the five
closed highly homogeneous groups; however $\E B\BBindex$ must be of
size $2$ whenever $\PM\BBindex$ is of the form $\Rev(R)$ and of size
$1$ whenever $\PM\BBindex$ is of the form $\Aut(R)$.

The definition of $\DataInv$ must be adjusted as well: if
$\PM\BBindex$ is of the form $\Rev(R)$ and therefore $B\BBindex$ is of
size $2$, then $\domain\BBindex$ consists of a single copy of the
support of $\PM\BBindex$. Furthermore, the diagonal action of an element $f$
of $F$ on $\domain$ must be adjusted: assumes that $\PM\BBindex$ is of
the form $\Rev(R)$ and therefore $B\BBindex$ is of size $2$; let $j'$
be such that $f$ maps $B\BBindex$ to $B\BBindex[j']$. Then, $f$ maps
the elements of $\domain\BBindex$ onto those of $\domain\BBindex[j']$,
with a reversal whenever the elements of $B\BBindex$ are swapped by
$f$ in $B\BBindex[j']$.

\begin{theorem}
  \label{theorem.classification}
  $P$-oligomorphic permutation groups are classified by finite
  permutation groups with (extended) decorated blocks through the
  $\E\Data$ and $\E\DataInv$ reciprocal correspondences.
\end{theorem}
\begin{proof}
  Follow the steps of the proof of
  Theorem~\ref{theorem.revfree_classification}, mutatis-mutandis.
\end{proof}

\section{Resolution of the conjectures and the Cohen-Macaulay property}
\label{section.consequences}

Here, we use the classification previously obtained to prove that 
the orbit algebra of $G$ is isomorphic to (a simple quotient of) the 
invariant ring of a finite
permutation group $F$. We deduce that Macpherson's conjecture holds:
$\orbitalgebra{G}$ is finitely generated, and even Cohen-Macaulay. In
addition, the minimal finite index subgroup $K$ prescribes the algebraic 
dimension of the orbit algebra
of $G$ and provides a natural system of parameters, and thus (a
choice of) the degrees appearing in the denominator of \ref{hilbert_series},
the Hilbert series.

Let $G$ be a (closed) $P$-oligomorphic group, $\fisubgroup = \prod_j K\BBindex$ 
its minimal subgroup of finite index, and use again the notations of 
the previous section.
Let $D_G$ be the set of degrees of the non zero degree elements of the 
ages $\age{H\BBindex}$ of the $H\BBindex$'s.
\begin{theorem}
  \label{theorem.inv_algebra}
  Let $G$ be a permutation group whose profile is bounded by a
  polynomial. Then, $\orbitalgebra{G}$ is isomorphic to the algebra of
  invariants of some finite permutation group acting on variables of
  degrees $D_G$, quotiented by the relations $x^2=0$ for some of the variables.
\end{theorem}

\begin{proof}
  For each \blockofblocks $\BB$, let $S_j$ be the collection of all
  the non-trivial subsets of all blocks of $\BB$. Let
  $S=\sqcup_j S_j$. By the definition of block systems, $K$ acts on
  each $S_j$ and on $S$. Denote by $(\theta_{i,j})_i$ the $K$-orbits
  in $S_j$ and observe that they are in bijection with the positive
  degree part $\age{G}^+(H\BBindex)$ of the age of $H\BBindex$.

  As in Example~\ref{example.wreath_products}, the orbit algebra
  $\orbitalgebra{K\BBindex}$ of $K\BBindex$, for $j \neq 0$, is the
  free algebra $\QQ[(\theta_{i,j})_i]$; for $j=0$, the orbit algebra
  is the finite dimensional algebra
  $\QQ[(\theta_{i,0})_i] / (\theta_{i,0}^2=0~ \forall i)$ instead. The
  orbit algebra of $K$ itself is the tensor product
  $\bigotimes_j\orbitalgebra{K\BBindex}$, generated by
  $(\theta_{i,j})_{i,j}$.

  The group $G$ itself also acts on $S$; since $K$ is normal in $G$,
  this lifts to an action on the finitely many $K$-orbits
  $(\theta_{i,j})_{i,j}$ in $S$. Let $G_0$ be the finite permutation
  group induced by this action, and let
  $\QQ[(\theta_{i,j})_{i,j}]^{G_0}$ be its invariant ring. Then,
  $\orbitalgebra{G}$ is the following quotient thereof:
  \begin{displaymath}
    \orbitalgebra{G} = \QQ[(\theta_{i,j})_{i,j}]^{G_0} / (\theta_{i,0}^2=0~ \forall i)\,.
  \end{displaymath}
\end{proof}

An immediate corollary is a positive resolution of Macpherson's
question and therefore Cameron's conjecture.
\begin{corollary}
  The orbit algebra $\orbitalgebra{G}$ of a $P$-oligomorphic
  permutation group $G$ is finitely generated.
\end{corollary}
\begin{proof}
  By Hilbert theorem, the invariant ring of a finite group is finitely
  generated (see e.g. \cite[Theorem 1.2]{Stanley.1979}) and $\orbitalgebra{G}$ is a
  quotient thereof.
\end{proof}

Invariant rings of permutation groups are not only finitely generated
and but also Cohen-Macaulay (see e.g. \cite[Theorem
3.2]{Stanley.1979}). The quotient of a Cohen-Macaulay algebra is not
Cohen-Macaulay in general, but the statement and proof of
\cite[Theorem 3.2]{Stanley.1979} can be generalized to suitable
quotients of invariant rings to deduce that orbit algebras are
Cohen-Macaulay.

Let us first recall the following classical fact about graded
commutative algebras.

\begin{fact}
  \label{fact.graded_free_module_generators}
  Let $R=\bigoplus_{d\in \NN} R_d$ be a graded connected commutative
  algebra (over a field $\mathbb{K}$) and take a free graded $R$-module $A$.
  Then, from any homogeneous $\mathbb{K}$-basis of $A$, we may
  extract a subfamily which is an $R$-basis of $A$.
\end{fact}

\begin{proof}
  Write $R^+=\bigoplus_{d>0} R_d$, and consider a family $(\nu_k)_{k\in K}$
  of secondary invariants for $A$: $A=\bigoplus_{k\in K} \nu_k R$\ ;
  let then $(\beta_j)_j$ be a vector basis of $A$, each $\beta_j$ being
  chosen homogeneous.
  Proceed by degree $d$: note that
  the elements $\nu_k$ of degree $d$ form a vector basis of a supplementary
  in $A_d$ of the homogeneous component $(R^+A)_d$ of the ideal
  $R^+A$; extract from $(\beta_j)_j$ a basis of some other supplementary. We
  obtain a subfamily $(\mu_k)_{k\in K}$ which spans $M$ as an
  $R$-module and is equinumerous in each degree with the family
  $(\nu_k)_k$. By dimension count, it is an $R$-basis of $M$.
\end{proof}

\begin{lemma}
  \label{lemma.cm.quotient}
  Let $F$ be a finite permutation group of a set $X\sqcup Y$ of
  variables. Assume that $F$ stabilizes $X$ and $Y$. Then,
  $\QQ[X\sqcup Y]^F / (y^2,\ y\in Y)$ is Cohen-Macaulay.
\end{lemma}
\begin{proof}
  Since $F$ stabilizes $X$ and $Y$, we may take $\sym(X)\sym(Y)$ as
  ring of primary invariants, over which $A=\QQ[X\sqcup Y]^F$ is a
  finite dimensional free module; in other words, there exists a finite 
  collection of
  homogeneous invariants $\nu_k$ such that we have:
  \begin{displaymath}
    A=\bigoplus_{k\in K} \nu_k \sym(X)\sym(Y)\,.
  \end{displaymath}
  It follows that $A$ is an (infinite dimensional) $\sym(X)$-free
  module. For instance, we have:
  \begin{displaymath}
    A=\bigoplus_{k,\lambda} \nu_k p_\lambda(Y) \sym(X)\,,
  \end{displaymath}
  where $(p_\lambda(Y))_{\lambda}$ is the \emph{power sum
    multiplicative basis} of $\sym(Y)$: for $\lambda$ an integer
  partition, $p_\lambda(Y)=\prod_{d\in \lambda}p_d(Y)$, where
  $p_d(Y):=\sum_{y\in Y} y^d$ is the $d$-th \emph{symmetric power sum}.
  A first consequence is that the quotient $A/(y^2,\ y\in Y)$ is a
  finitely generated $\sym(X)$-module:
  \begin{displaymath}
    A / (y^2,\ y\in Y)=\sum_{k\in K,\ i=1,\dots,|Y|} \nu_k p_1^i(Y) \sym(X)\,.
  \end{displaymath}

  There remains to show that it is a free module. To this end, let us
  come back to $A$ and consider its monomial basis, that is, the
  collection of all $F$-orbit sums of monomials in $\QQ[X\sqcup Y]$.
  By Fact~\ref{fact.graded_free_module_generators}, we may extract
  from the monomial basis a family $(\mu_\ell)_{\ell\in L}$ which forms a
  $\sym(X)$-basis of $A$:
  \begin{displaymath}
    A=\bigoplus_{\ell \in L} \mu_\ell \sym(X)\,.
  \end{displaymath}
  Remark that, if a monomial is divisible by $y^2$ for some $y\in Y$,
  then the same holds for any other monomial in its orbit sum.
  Therefore, in the quotient, each monomial $\mu_\ell$ either vanishes
  completely or is left unaffected, and the same holds accordingly for
  the principal ideal $\mu_\ell \sym(X)$. Therefore, we have indeed:
  \begin{displaymath}
    A / (y^2,\  y\in Y)=\bigoplus_{\ell \in L'} \mu_\ell \sym(X)\,,
  \end{displaymath}
  for some subset $L'$ of $L$, so that $A / (y^2)$ is a finite dimensional
  free $\sym(X)$-module, as desired.
\end{proof}

\begin{corollary}
  \label{corollary.Cohen-Macaulay}
  The orbit algebra $\orbitalgebra{G}$ of a $P$-oligomorphic
  permutation group $G$ is Cohen-Macaulay.
\end{corollary}

\newpage
\bibliographystyle{alpha}
\bibliography{main}

\end{document}

%% file: figures/figure_poly_sym.tex
\begin{figure}[h]
\begin{center}
\quad
\raisebox{2.5cm}{\text{\large{$\sg_\infty$}}} \,
\raisebox{0.07cm}{
\begin{tikzpicture}
  \coordinate (P1) at ($(0, 0) + (-200:0.3cm and 2.4cm)$);
  \draw[>=stealth,<-] (P1) arc (-200:130:0.3cm and 2.4cm);
\end{tikzpicture} }
\quad ~
\begin{tikzpicture}
  \coordinate (P1) at ($(0, 0) + (10:0.4cm and 2.5cm)$);
  \coordinate (P2) at ($(0, 0) + (-10:0.4cm and 2.5cm)$);
  \coordinate (P3) at ($(0, 0) + (170:0.4cm and 2.5cm)$);
  \coordinate (P4) at ($(0, 0) + (190:0.4cm and 2.5cm)$);
  
  \draw[thick] (P1) arc (10:170:0.4cm and 2.5cm);
  \draw[dashed,thick] (P2) arc (-10:10:0.4cm and 2.5cm);
  \draw[dashed,thick] (P3) arc (170:190:0.4cm and 2.5cm);
  \draw[thick] (P4) arc (190:350:0.4cm and 2.5cm);
  \draw[draw opacity = 0, fill = black!10] (0, 0) ellipse [x radius = 0.4cm, y radius = 2.5cm];

  \draw (0,-2.3) node[circle,draw=black,fill=blue,scale=0.5] {};
  \draw (0,-2) node[circle,draw=black,fill=blue,scale=0.5] {};
\end{tikzpicture}
\quad
\begin{tikzpicture}
  \coordinate (P1) at ($(0, 0) + (10:0.4cm and 2.5cm)$);
  \coordinate (P2) at ($(0, 0) + (-10:0.4cm and 2.5cm)$);
  \coordinate (P3) at ($(0, 0) + (170:0.4cm and 2.5cm)$);
  \coordinate (P4) at ($(0, 0) + (190:0.4cm and 2.5cm)$);
  
  \draw[thick] (P1) arc (10:170:0.4cm and 2.5cm);
  \draw[dashed,thick] (P2) arc (-10:10:0.4cm and 2.5cm);
  \draw[dashed,thick] (P3) arc (170:190:0.4cm and 2.5cm);
  \draw[thick] (P4) arc (190:350:0.4cm and 2.5cm);
  \draw[draw opacity = 0, fill = black!10] (0, 0) ellipse [x radius = 0.4cm, y radius = 2.5cm];
    
  \node[circle,draw=black,fill=magenta,scale=0.5] (a1) at (-0.2,1){};
  \draw (0,-2.3) node[circle,draw=black,fill=blue,scale=0.5] {};
  %\node[circle,draw=black,fill=blue] (b2) at (0,-2.3){};
\end{tikzpicture}
\quad
\begin{tikzpicture}
  \coordinate (P1) at ($(0, 0) + (10:0.4cm and 2.5cm)$);
  \coordinate (P2) at ($(0, 0) + (-10:0.4cm and 2.5cm)$);
  \coordinate (P3) at ($(0, 0) + (170:0.4cm and 2.5cm)$);
  \coordinate (P4) at ($(0, 0) + (190:0.4cm and 2.5cm)$);
  
  \draw[thick] (P1) arc (10:170:0.4cm and 2.5cm);
  \draw[dashed,thick] (P2) arc (-10:10:0.4cm and 2.5cm);
  \draw[dashed,thick] (P3) arc (170:190:0.4cm and 2.5cm);
  \draw[thick] (P4) arc (190:350:0.4cm and 2.5cm);
  \draw[draw opacity = 0, fill = black!10] (0, 0) ellipse [x radius = 0.4cm, y radius = 2.5cm];

  \draw (0,-2.3) node[circle,draw=black,fill=blue,scale=0.5] {};
\end{tikzpicture}
\quad
\begin{tikzpicture}
  \coordinate (P1) at ($(0, 0) + (10:0.4cm and 2.5cm)$);
  \coordinate (P2) at ($(0, 0) + (-10:0.4cm and 2.5cm)$);
  \coordinate (P3) at ($(0, 0) + (170:0.4cm and 2.5cm)$);
  \coordinate (P4) at ($(0, 0) + (190:0.4cm and 2.5cm)$);
  
  \draw[thick] (P1) arc (10:170:0.4cm and 2.5cm);
  \draw[dashed,thick] (P2) arc (-10:10:0.4cm and 2.5cm);
  \draw[dashed,thick] (P3) arc (170:190:0.4cm and 2.5cm);
  \draw[thick] (P4) arc (190:350:0.4cm and 2.5cm);
  \draw[draw opacity = 0, fill = black!10] (0, 0) ellipse [x radius = 0.4cm, y radius = 2.5cm];

  \draw (-0.15,0.5) node[circle,draw=black,fill=magenta,scale=0.5] {};
  \draw (0.2,1.3) node[circle,draw=black,fill=magenta,scale=0.5] {};
\end{tikzpicture}
\quad
\begin{tikzpicture}
  \coordinate (P1) at ($(0, 0) + (10:0.4cm and 2.5cm)$);
  \coordinate (P2) at ($(0, 0) + (-10:0.4cm and 2.5cm)$);
  \coordinate (P3) at ($(0, 0) + (170:0.4cm and 2.5cm)$);
  \coordinate (P4) at ($(0, 0) + (190:0.4cm and 2.5cm)$);
  
  \draw[thick] (P1) arc (10:170:0.4cm and 2.5cm);
  \draw[dashed,thick] (P2) arc (-10:10:0.4cm and 2.5cm);
  \draw[dashed,thick] (P3) arc (170:190:0.4cm and 2.5cm);
  \draw[thick] (P4) arc (190:350:0.4cm and 2.5cm);
  \draw[draw opacity = 0, fill = black!10] (0, 0) ellipse [x radius = 0.4cm, y radius = 2.5cm];

  \draw (0.1,-1.6) node[circle,draw=black,fill=magenta,scale=0.5] {};
\end{tikzpicture}
\qquad \quad \quad \raisebox{2.5cm}{\text{\large{$\sg_\infty \wreath \sg_5$}}}

\vspace{0.4cm}
\begin{tikzpicture}
  \coordinate (P1) at ($(0, 0) + (-40:2.9cm and 0.35cm)$);
  \draw[>=stealth,->] (P1) arc (-40:290:2.9cm and 0.35cm);
\end{tikzpicture}

\vspace{0.2cm}

\text{\large{$\sg_5$}}
\end{center}
\caption{An example of a wreath product: $\sg_\infty \wreath \sg_5$;
  the two highlighted subsets, in red and blue respectively, are in
  the same orbit.}
\label{figure.wreath_infinite_blocks}
\end{figure}

%% file: figures/figure_partitions_blocks_finis.tex
\begin{figure}[h]
\begin{center}
\raisebox{1cm}{\text{\large{$\sg_k$}}} \quad
\begin{tikzpicture}
  \coordinate (P1) at ($(0, 0) + (-200:0.3cm and 1cm)$);
  \draw[>=stealth,<-] (P1) arc (-200:130:0.3cm and 1cm);
\end{tikzpicture}
\qquad
\begin{tikzpicture}
  \coordinate (P1) at ($(0, 0) + (10:0.4cm and 1cm)$);
  \draw[thick] (P1) arc (10:370:0.4cm and 1cm);
  \draw[draw opacity = 0, fill = black!10] (0, 0) ellipse [x radius = 0.4cm, y radius = 1cm];
  \draw (0,-0.8) node[circle,draw=black,fill=blue,scale=0.5] {};
  \draw (0,-0.55) node[circle,draw=black,fill=blue,scale=0.5] {};
  \draw (0,-0.30) node[circle,draw=black,fill=blue,scale=0.5] {};
\end{tikzpicture}
\,
\begin{tikzpicture}
  \coordinate (P1) at ($(0, 0) + (10:0.4cm and 1cm)$);
  \draw[thick] (P1) arc (10:370:0.4cm and 1cm);
  \draw[draw opacity = 0, fill = black!10] (0, 0) ellipse [x radius = 0.4cm, y radius = 1cm];
  \draw (0,-0.8) node[circle,draw=black,fill=blue,scale=0.5] {};
\end{tikzpicture}
\,
\begin{tikzpicture}
  \coordinate (P1) at ($(0, 0) + (10:0.4cm and 1cm)$);
  \draw[thick] (P1) arc (10:370:0.4cm and 1cm);
  \draw[draw opacity = 0, fill = black!10] (0, 0) ellipse [x radius = 0.4cm, y radius = 1cm];
  \draw (-0.1,0.5) node[circle,draw=black,fill=magenta,scale=0.5] {};
\end{tikzpicture}
\,
\begin{tikzpicture}
  \coordinate (P1) at ($(0, 0) + (10:0.4cm and 1cm)$);
  \draw[thick] (P1) arc (10:370:0.4cm and 1cm);
  \draw[draw opacity = 0, fill = black!10] (0, 0) ellipse [x radius = 0.4cm, y radius = 1cm];
\end{tikzpicture}
\,
\begin{tikzpicture}
  \coordinate (P1) at ($(0, 0) + (10:0.4cm and 1cm)$);
  \draw[thick] (P1) arc (10:370:0.4cm and 1cm);
  \draw[draw opacity = 0, fill = black!10] (0, 0) ellipse [x radius = 0.4cm, y radius = 1cm];
  \draw (0.1,0.3) node[circle,draw=black,fill=magenta,scale=0.5] {};
  \draw (-0.05,-0.4) node[circle,draw=black,fill=magenta,scale=0.5] {};
  \draw (0,0.75) node[circle,draw=black,fill=magenta,scale=0.5] {};
\end{tikzpicture}
\quad \raisebox{1cm}{\ldots}

\vspace{0.3cm}
\qquad \qquad \qquad
\begin{tikzpicture}
  \coordinate (P1) at ($(0, 0) + (-40:3.5cm and 0.35cm)$);
  \draw[>=stealth,->] (P1) arc (-40:290:3.5cm and 0.35cm);
\end{tikzpicture}

\vspace{0.2cm}
\hspace{2cm}
\text{\large{$\sg_\infty$}}
\end{center}
\caption{Example of a wreath product: $\sg_k \wreath \sg_\infty$, with two subsets
in the same orbit}
\label{figure.wreath_finite_blocks}
\end{figure}

%% file: figures/figure_cas_essentiels.tex
\begin{figure}[h]
\begin{center}
\begin{tikzpicture}
  \coordinate (P1) at ($(0, 0) + (10:0.4cm and 2.5cm)$);
  \coordinate (P2) at ($(0, 0) + (-10:0.4cm and 2.5cm)$);
  \coordinate (P3) at ($(0, 0) + (170:0.4cm and 2.5cm)$);
  \coordinate (P4) at ($(0, 0) + (190:0.4cm and 2.5cm)$);
  
  \draw[thick] (P1) arc (10:170:0.4cm and 2.5cm);
  \draw[dashed,thick] (P2) arc (-10:10:0.4cm and 2.5cm);
  \draw[dashed,thick] (P3) arc (170:190:0.4cm and 2.5cm);
  \draw[thick] (P4) arc (190:350:0.4cm and 2.5cm);
  \draw[draw opacity = 0, fill = black!10] (0, 0) ellipse [x radius = 0.4cm, y radius = 2.5cm];
\end{tikzpicture}
\quad
\begin{tikzpicture}
  \coordinate (P1) at ($(0, 0) + (10:0.4cm and 2.5cm)$);
  \coordinate (P2) at ($(0, 0) + (-10:0.4cm and 2.5cm)$);
  \coordinate (P3) at ($(0, 0) + (170:0.4cm and 2.5cm)$);
  \coordinate (P4) at ($(0, 0) + (190:0.4cm and 2.5cm)$);
  
  \draw[thick] (P1) arc (10:170:0.4cm and 2.5cm);
  \draw[dashed,thick] (P2) arc (-10:10:0.4cm and 2.5cm);
  \draw[dashed,thick] (P3) arc (170:190:0.4cm and 2.5cm);
  \draw[thick] (P4) arc (190:350:0.4cm and 2.5cm);
  \draw[draw opacity = 0, fill = black!10] (0, 0) ellipse [x radius = 0.4cm, y radius = 2.5cm];
\end{tikzpicture}
\quad \raisebox{2.6cm}{\ldots}
\quad
\begin{tikzpicture}
  \coordinate (P1) at ($(0, 0) + (10:0.4cm and 2.5cm)$);
  \coordinate (P2) at ($(0, 0) + (-10:0.4cm and 2.5cm)$);
  \coordinate (P3) at ($(0, 0) + (170:0.4cm and 2.5cm)$);
  \coordinate (P4) at ($(0, 0) + (190:0.4cm and 2.5cm)$);
  
  \draw[thick] (P1) arc (10:170:0.4cm and 2.5cm);
  \draw[dashed,thick] (P2) arc (-10:10:0.4cm and 2.5cm);
  \draw[dashed,thick] (P3) arc (170:190:0.4cm and 2.5cm);
  \draw[thick] (P4) arc (190:350:0.4cm and 2.5cm);
  \draw[draw opacity = 0, fill = black!10] (0, 0) ellipse [x radius = 0.4cm, y radius = 2.5cm];
\end{tikzpicture}
\quad
\begin{tikzpicture}
  \coordinate (P1) at ($(0, 0) + (10:0.4cm and 2.5cm)$);
  \coordinate (P2) at ($(0, 0) + (-10:0.4cm and 2.5cm)$);
  \coordinate (P3) at ($(0, 0) + (170:0.4cm and 2.5cm)$);
  \coordinate (P4) at ($(0, 0) + (190:0.4cm and 2.5cm)$);
  
  \draw[thick] (P1) arc (10:170:0.4cm and 2.5cm);
  \draw[dashed,thick] (P2) arc (-10:10:0.4cm and 2.5cm);
  \draw[dashed,thick] (P3) arc (170:190:0.4cm and 2.5cm);
  \draw[thick] (P4) arc (190:350:0.4cm and 2.5cm);
  \draw[draw opacity = 0, fill = black!10] (0, 0) ellipse [x radius = 0.4cm, y radius = 2.5cm];
\end{tikzpicture}
\qquad \qquad \qquad
\raisebox{1.6cm}{
\begin{tikzpicture}
  \coordinate (P1) at ($(0, 0) + (10:0.4cm and 1cm)$);
  \draw[thick] (P1) arc (10:370:0.4cm and 1cm);
  \draw[draw opacity = 0, fill = black!10] (0, 0) ellipse [x radius = 0.4cm, y radius = 1cm];
\end{tikzpicture}}
\,
\raisebox{1.6cm}{
\begin{tikzpicture}
  \coordinate (P1) at ($(0, 0) + (10:0.4cm and 1cm)$);
  \draw[thick] (P1) arc (10:370:0.4cm and 1cm);
  \draw[draw opacity = 0, fill = black!10] (0, 0) ellipse [x radius = 0.4cm, y radius = 1cm];
\end{tikzpicture}}
\,
\raisebox{1.6cm}{
\begin{tikzpicture}
  \coordinate (P1) at ($(0, 0) + (10:0.4cm and 1cm)$);
  \draw[thick] (P1) arc (10:370:0.4cm and 1cm);
  \draw[draw opacity = 0, fill = black!10] (0, 0) ellipse [x radius = 0.4cm, y radius = 1cm];
\end{tikzpicture}}
\,
\raisebox{1.6cm}{
\begin{tikzpicture}
  \coordinate (P1) at ($(0, 0) + (10:0.4cm and 1cm)$);
  \draw[thick] (P1) arc (10:370:0.4cm and 1cm);
  \draw[draw opacity = 0, fill = black!10] (0, 0) ellipse [x radius = 0.4cm, y radius = 1cm];
\end{tikzpicture}}
\quad \raisebox{2.6cm}{\ldots}

\hspace{-.5cm}
\raisebox{-0.2cm}{\text{Case 1}} \hspace{7.3cm} 
\raisebox{-0.2cm}{\text{Case 2}}
\end{center}
\caption{Essential cases for a \emph{transitive} block system of a $P$-oligomorphic group}
\end{figure}

%% file: figures/figure_proof_finite_blocks1.tex
  
    
\begin{figure}[h]\begin{center}
\begin{tikzpicture}
\node[circle,draw=black,fill=white] (02) at (0,2){};
\node[circle,draw=black,fill=white] (12) at (1.5,2){};
\node[circle,draw=black,fill=white] (22) at (3,2){};
\node[circle,draw=black,fill=gray!55] (01) at (0,1){};
\node[circle,draw=black,fill=gray!55] (11) at (1.5,1){};
\node[circle,draw=black,fill=gray!55] (21) at (3,1){};
\node[circle,draw=black,fill=black] (00) at (0,0){};
\node[circle,draw=black,fill=black] (10) at (1.5,0){};
\node[circle,draw=black,fill=black] (20) at (3,0){};

\draw[>=stealth,magenta,->] (02) -- (12);
\draw[>=stealth,magenta,->] (01) -- (11);
\draw[>=stealth,magenta,->] (00) -- (10);
\draw[>=stealth,cyan,<-] (12) -- (22);
\draw[>=stealth,cyan,<-] (11) -- (21);
\draw[>=stealth,cyan,<-] (10) -- (20);

\draw[>=stealth,violet,<-] (02) -- (10);
\draw[>=stealth,violet,<-] (01) -- (12);
\draw[>=stealth,violet,<-] (00) -- (11);
\draw[>=stealth,blue,->] (10) -- (22);
\draw[>=stealth,blue,->] (11) -- (20);
\draw[>=stealth,blue,->] (12) -- (21);

\draw (0,-1) node {$B_i$};
\draw (1.5,-1) node {$B_1$};
\draw (3,-1) node {$B_j$};

\draw (0.75,2.5) node {$\tau_{1,i}$};
\draw (2.25,2.5) node {$\tau_{1,j}$};
\end{tikzpicture}
\qquad \qquad
\raisebox{2.2cm}{$\underrightarrow{~\color{magenta}{\tau_{1,i}}
\color{blue}{\tau_{1,j}^{-1}} \color{magenta}{\tau_{1,i}}}$}
\qquad \qquad
\begin{tikzpicture}
\node[circle,draw=black,fill=white] (02) at (0,2){};
\node[circle,draw=black,fill=black] (12) at (1.5,2){};
\node[circle,draw=black,fill=white] (22) at (3,2){};
\node[circle,draw=black,fill=gray!55] (01) at (0,1){};
\node[circle,draw=black,fill=white] (11) at (1.5,1){};
\node[circle,draw=black,fill=gray!55] (21) at (3,1){};
\node[circle,draw=black,fill=black] (00) at (0,0){};
\node[circle,draw=black,fill=gray!55] (10) at (1.5,0){};
\node[circle,draw=black,fill=black] (20) at (3,0){};

\draw (0,-1) node {\color{orange}{$B_j$}};
\draw (1.5,-1) node {$B_1$};
\draw (3,-1) node {\color{orange}{$B_i$}};

\draw (0.5,2.5) node {};
\draw (1.5,2.5) node {};

\end{tikzpicture}
\end{center}
\caption{Example: Straight swap of $B_i$ and $B_j$ (also permuting $B_1$)}
\label{figure.proof_finite_blocks1}
\end{figure}
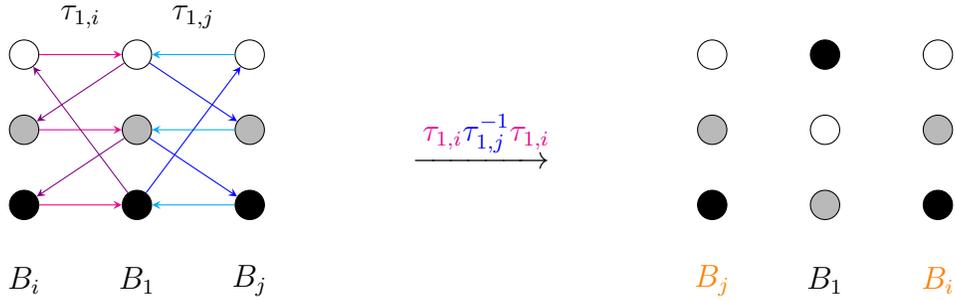

%% file: figures/figure_proof_finite_blocks2.tex
  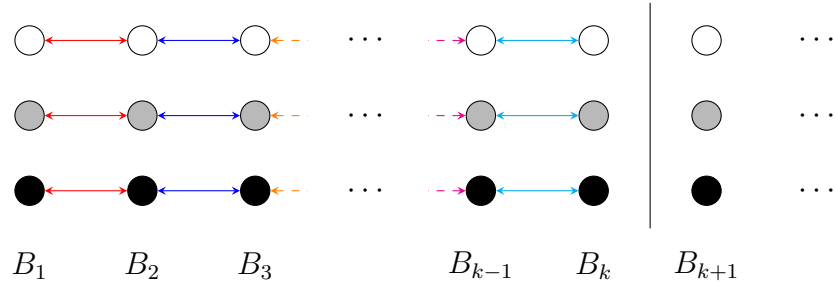
\begin{figure}[h] \begin{center}
  \begin{tikzpicture}

\node[circle,draw=black,fill=white] (02) at (0,2){};
\node[circle,draw=black,fill=white] (12) at (1.5,2){};
\node[circle,draw=black,fill=white] (22) at (3,2){};
\node[circle,draw=black,fill=gray!55] (01) at (0,1){};
\node[circle,draw=black,fill=gray!55] (11) at (1.5,1){};
\node[circle,draw=black,fill=gray!55] (21) at (3,1){};
\node[circle,draw=black,fill=black] (00) at (0,0){};
\node[circle,draw=black,fill=black] (10) at (1.5,0){};
\node[circle,draw=black,fill=black] (20) at (3,0){};
\draw (4.5,2) node {$\cdots$};
\draw (4.5,1) node {$\cdots$};
\draw (4.5,0) node {$\cdots$};
\node[circle,draw=black,fill=white] (42) at (6,2){};
\node[circle,draw=black,fill=gray!55] (41) at (6,1){};
\node[circle,draw=black,fill=black] (40) at (6,0){};
\node[circle,draw=black,fill=white] (52) at (7.5,2){};
\node[circle,draw=black,fill=gray!55] (51) at (7.5,1){};
\node[circle,draw=black,fill=black] (50) at (7.5,0){};
\draw (8.25,-0.5) -- (8.25,2.5);
\node[circle,draw=black,fill=white] (62) at (9,2){};
\node[circle,draw=black,fill=gray!55] (61) at (9,1){};
\node[circle,draw=black,fill=black] (60) at (9,0){};
\draw (10.5,2) node {$\cdots$};
\draw (10.5,1) node {$\cdots$};
\draw (10.5,0) node {$\cdots$};

\draw (0,-1) node {$B_1$};
\draw (1.5,-1) node {$B_2$};
\draw (3,-1) node {$B_3$};
\draw (6,-1) node {$B_{k-1}$};
\draw (7.5,-1) node {$B_{k}$};
\draw (9,-1) node {$B_{k+1}$};

\draw[>=stealth,red,<->] (02) -- (12);
\draw[>=stealth,red,<->] (01) -- (11);
\draw[>=stealth,red,<->] (00) -- (10);
\draw[>=stealth,blue,<->] (22) -- (12);
\draw[>=stealth,blue,<->] (21) -- (11);
\draw[>=stealth,blue,<->] (20) -- (10);
\draw[>=stealth,orange,dashed,<-] (22) -- (3.7,2);
\draw[>=stealth,orange,dashed,<-] (21) -- (3.7,1);
\draw[>=stealth,orange,dashed,<-] (20) -- (3.7,0);
\draw[>=stealth,magenta,dashed,<-] (42) -- (5.3,2);
\draw[>=stealth,magenta,dashed,<-] (41) -- (5.3,1);
\draw[>=stealth,magenta,dashed,<-] (40) -- (5.3,0);
\draw[>=stealth,cyan,<->] (42) -- (52);
\draw[>=stealth,cyan,<->] (41) -- (51);
\draw[>=stealth,cyan,<->] (40) -- (50);
\end{tikzpicture}
\end{center}
\caption{Straight swaps between the first $k$ blocks}
\label{figure.proof_finite_blocks2}
\end{figure}